\documentclass[reqno,12pt]{amsproc}
\usepackage{amsbsy}
\usepackage{amsfonts}
\usepackage{amsmath}
\usepackage{amsthm}
\usepackage{amssymb}
\usepackage{enumitem}

\usepackage{hyperref}
\usepackage{color}
\usepackage{etoolbox}
\expandafter\patchcmd\csname\string\proof\endcsname
  {\normalparindent}{0pt }{}{}

\usepackage[margin=1in]{geometry}

\theoremstyle{plain}

\newtheorem{lemma}{Lemma}
\newtheorem{proposition}{Proposition}
\newtheorem{theorem}{Theorem}

\newtheorem{conjecture}{Conjecture}
\newtheorem{remark}{Remark}

\newcommand{\Mod}[1]{\, (\text{mod}\ #1)}

\newcommand{\Res}{\mathop{\mathrm{Res}}}
\newcommand{\sumd}{\mathop{\sum\,^d}}
\newcommand{\mm}{\mathfrak{m}}
\newcommand{\nn}{\mathfrak{n}}

\begin{document}

\title{The variance of divisor sums in arithmetic progressions}
\author{Brad Rodgers, Kannan Soundararajan}
\address{Department of Mathematics, University of Michigan, Ann Arbor, MI 48109}
\email{rbrad@umich.edu}
\address{Department of Mathematics, Stanford University, Stanford CA 94305}
\email{ksound@stanford.edu} 
\date{}

\begin{abstract}
We study the variance of sums of the $k$-fold divisor function $d_k(n)$ over sparse arithmetic progressions, with averaging over both residue classes and moduli. In a restricted range, we confirm an averaged version of a recent conjecture about the asymptotics of this variance. This result is closely related to moments of Dirichlet $L$-functions and our proof relies on the asymptotic large sieve.
\end{abstract}


\maketitle

\section{Introduction}
\label{sec:intro}

Consider the $k$-fold divisor function,
$$
d_k(n):= \#\{(a_1,...,a_k) \in \mathbb{N}^k:\, a_1\cdots a_k = n\},
$$
which has the Dirichlet series
\begin{equation}
\label{DivisorGen}
\sum_n \frac{d_k(n)}{n^s} = \zeta^k(s).
\end{equation}
For $(b,q)=1$, the extent to which the sums
\begin{equation}
\label{APsum}
\sum_{\substack{n \equiv b \Mod q \\ n \leq X}} d_k(n)
\end{equation}
are approximated by the mean value
$$
\frac{1}{\phi(q)}\sum_{\substack{(n,q)=1 \\ n \leq X}} d_k(n),
$$
has been extensively studied. To mention examples from a long line of work, Heath-Brown \cite{He1}, Hooley \cite{Ho}, and Friedlander and Iwaniec \cite{FrIw} have obtained uniform estimates for various $k$. Our purpose will be to study the variance of these sums as $b$ varies. In the case that $k=2$ this problem has been studied in various ways (see e.g. \cite{Bl, Mo, FoGaKoMi}) with in particular Lau and Zhao \cite{LaZh} proving a pleasant and simple asymptotic formula for the variance with $q$ and $X$ growing at different rates.

In the case that $k$ is larger, a function field variant of this problem has recently been considered by Keating, the first author, Roditty-Gershon, and Rudnick \cite{KeRoRoRu}. Their result suggests the following conjecture over the integers.
\begin{conjecture}
\label{APConj}
For $X,q\rightarrow\infty$ in such a way that $\frac{\log X}{\log q} \rightarrow c \in (0,k)$, we have
\begin{align}
\label{sharp_variance}
v_k(q;X) &:= \sum_{\substack{1\leq a \leq q \\ (a,q)=1}} \Big| \sum_{\substack{n \equiv a \Mod q \\ n \leq X}} d_k(n) - \frac{1}{\phi(q)} \sum_{\substack{(n,q)=1 \\ n \leq X}} d_k(n)\Big|^2\\
&\notag \sim a_k(q) \gamma_k(c) X(\log q)^{k^2-1},
\end{align}
where $a_k(q)$ is the arithmetic constant
\begin{equation}
\label{constant_q}
a_k(q):= \lim_{s\rightarrow 1^+} (s-1)^{k^2} \sum_{\substack{n\geq 1 \\ (n,q)=1}} \frac{d_k(n)^2}{n^s},
\end{equation}
and $\gamma_k(c)$ is a piecewise polynomial of degree $k^2-1$ that is positive for $c \in (0,k)$ and is described in more detail below.
\end{conjecture}

Here and in the rest of the paper $k \geq 2$ is assumed to be fixed.  It may be helpful to note that in this conjecture 
the range $c<1$ corresponds to $X \le q$ so that the arithmetic progressions in question have at most one term; the 
range when $c$ is just a little larger than $1$ corresponds to arithmetic progressions with large moduli (close to, but smaller than, $X$).   

In that paper an analogous conjecture is also made for the variance of divisor sums in short intervals:
\begin{conjecture}
\label{ShortConj}
For $X, H \geq 1$ such that $X\rightarrow\infty$ and $X/H\rightarrow\infty$ in such a way that $\frac{\log X}{\log (X/H)} \rightarrow c \in (0,k)$, we have
\begin{equation}
\label{short_int_variance}
\frac{1}{X} \int_X^{2X} \Big|\sum_{x \leq n \leq x+H} d_k(n)\Big|^2\,dx - \Big(\frac{1}{X} \int_X^{2X} \sum_{x\leq n \leq x+H} d_k(n)\, dx\Big)^2 \sim a_k \gamma_k(c) X(\log (X/H))^{k^2-1},
\end{equation}
where $a_k:= \lim (s-1)^{k^2} \sum d_k(n)^2 n^{-s}$. 
\end{conjecture}
Note that $X$ plays the same role in \eqref{short_int_variance} and \eqref{sharp_variance}, while $H$ in \eqref{short_int_variance} plays the role of the number of terms $X/q$ in each summand of \eqref{sharp_variance}.

The piecewise polynomial $\gamma_k(c)$ appearing in these conjectures may be defined by
\begin{equation}
\label{piecewise_poly}
\gamma_k(c):= \frac{1}{k! \, G(k+1)^2} \int_{[0,1]^k} \delta_c(w_1+\cdots+w_k) \Delta(w)^2\, d^k w,
\end{equation}
with $\delta_c(x) := \delta(x-c)$ a Dirac delta-function centered at $c$, and $\Delta(w) := \prod_{i<j} (w_i-w_j)$ a Vandermonde determinant, and $G$ the Barnes $G$-function, so that in particular $G(k+1) = (k-1)!\cdot (k-2)! \cdots 1!$. Thus defined $\gamma_k(c)$ is a certain polynomial on each of the intervals $[0,1)$, $[1,2)$,...,$[k-1,k)$, with changes in its coefficients from interval to interval. For instance 
$$
\gamma_3(c) = \begin{cases} 
 \frac{1}{8!}c^8 &\textrm{for}\;0\leq c<1,\\
 \frac{1}{8!} (3-c)^8 &\textrm{for} \; 2\le c<3,  
 \end{cases}
$$
while for $1\le c\le 2$ we have 
$$
\gamma_3(c) = \frac 1{8!}\Big(-2 c^8+24 c^7-252 c^6+1512 c^5\\-4830
c^4+8568 c^3-8484 c^2+4392 c-927\Big).
$$
In general
$$
\gamma_k(c) = \gamma_k(k-c).
$$
That $\gamma_k(c)$ changes from interval to interval in Conjectures \ref{APConj} and \ref{ShortConj} is an at first surprising phenomenon. Though analogous phase changes occur in the function field analogue proved in \cite{KeRoRoRu}, over the integers these phase changes remain somewhat mysterious. Indeed, the observation that there is somewhat strange behavior for limiting functions like $\gamma_k(c)$ may be said to date back at least to the work of Conrey and Gonek \cite{CoGo}, who studied polynomials very closely related to $\gamma_k(c)$ with $c \in [1,2)$ in their work on the eighth moment of the Riemann zeta-function. In fact, Conjectures \ref{APConj} and \ref{ShortConj} remain closely related to the moments of Dirichlet $L$-functions and the Riemann zeta-function respectively. See \cite{CoKe1,CoKe2,CoKe3,CoKe4} for recent heuristic work that is related to Conrey and Gonek's.

What is known rigorously over the integers in the short-interval setting of Conjecture \ref{ShortConj} follows from using summation formulas related to the functional equation for the Riemann zeta function. In this way Lester \cite{Le} has evaluated the variance for $c \in (k-1,k)$. It is likely that a similar argument could be used to verify Conjecture \ref{APConj} in this restricted range for all $k$ (indeed, this is close to the strategy of \cite{LaZh} in the case $k=2$).

Our purpose here is to demonstrate that a different range of the asymptotic evaluation $\gamma_k(c)$ in Conjecture \ref{APConj} may be rigorously seen if we allow ourselves to average over moduli $q$. We make use of smoothed weights as opposed to the sharp cutoffs in the variance \eqref{sharp_variance}, and our main result concerns the following quantities:

We let $\Phi$ and $\Psi$ be fixed smooth non-negative functions compactly supported in the positive reals normalized so that
$$
\int \Phi = 1, \quad\quad \int \Psi^2 = 1.
$$
Define
$$
V_k(q;X):= \sum_{\substack{1 \leq a \leq q \\ (a,q)=1}} \bigg| \sum_{n \equiv a \Mod q} d_k(n) \Psi\Big(\frac{n}{X}\Big) - \frac{1}{\phi(q)} \sum_{(n,q)=1} d_k(n) \Psi\Big(\frac{n}{X}\Big)\bigg|^2,
$$
and
$$
\Delta_k(Q;X):= \sum_q V_k(q;X) \Phi\Big(\frac{q}{Q}\Big).
$$
\begin{theorem}
\label{main_thm}
Fix $k \geq 2$. Suppose $X,Q\rightarrow \infty$ and introduce the parameter $c:=\frac{\log X}{\log Q}$. For $\delta > 0$, uniformly for $c$ constrained by $\delta \leq c \leq \frac{k+2}{k}-\delta,$ we have
\begin{align}
\label{smooth_variance1}
\Delta_k(Q;X) &\sim \sum_q a_k(q) X \gamma_k(c) (\log q)^{k^2-1} \Phi\Big(\frac{q}{Q}\Big) \\
\label{smooth_variance2}&\sim \widetilde{a}_k\gamma_k(c)  QX (\log Q)^{k^2-1},
\end{align}
where $\widetilde{a}_k$ is the arithmetic constant
\begin{equation}
\label{constant_avg}
\widetilde{a}_k:= \lim_{Q\rightarrow\infty} \frac{1}{Q}\sum_{q\leq Q} a_k(q).
\end{equation}
\end{theorem}

The constants $a_k(q)$ and $\widetilde{a}_k$ in \eqref{constant_q} and \eqref{constant_avg} may also be expressed as an Euler product. For
$$
\mathfrak{a}_p:= \sum_{\ell=0}^\infty \binom{k+\ell-1}{k-1}^2 p^{-\ell},\quad \textrm{and} \quad a_k:= \prod_p \Big(1-\frac{1}{p}\Big)^{k^2} \mathfrak{a}_p,
$$
we have
\begin{equation}
\label{constants_euler}
a_k(q) = a_k \prod_{p|q} \mathfrak{a}_p^{-1},\quad\textrm{and}\quad \widetilde{a}_k= a_k \prod_p\Big( 1 - \frac{1}{p} (1-\mathfrak{a}_p^{-1})\Big).
\end{equation}
We leave it to the reader to verify that this representation \eqref{constants_euler} for the arithmetic constants is the same as \eqref{constant_q} and \eqref{constant_avg}, and likewise that the expressions \eqref{smooth_variance1} and \eqref{smooth_variance2} in Theorem \ref{main_thm} are the same (owing to the slow growth of the logarithm function).

The work of de la Bret\`eche and Fiorilli \cite{BrFi} considers a related variance, using however an arithmetic approximation 
(motivated by work of Vaughan) instead of the probabilistic variance considered here.   Interestingly, their asymptotic for their arithmetic variance matches our result in Theorem 1.


If we assume the Generalized Riemann Hypothesis (GRH), we can prove a stronger result\footnote{In fact, a generalized Lindel\"of hypothesis will suffice.}:
\begin{theorem}
\label{main_thm_GRH}
On GRH, \eqref{smooth_variance1} is true uniformly for $\delta \leq c \leq 2-\delta$, for all $k \geq 2$.
\end{theorem}

We note that with more work and bookkeeping, the asymptotic formula in Theorems \ref{main_thm} and \ref{main_thm_GRH} should be able to be replaced by a more complicated expression with a power-saving error term. Likewise, in these theorems we have made use of smooth cut-offs $\Phi$ and $\Psi$. It remains an interesting challenge to prove results like these with sharp cut-offs.

In recent work of Harper and the second author \cite{HaSo}, it was shown unconditionally that a quantity essentially the same as $\Delta_k(Q;X)$ is at least of order $QX (\log Q)^{k^2-1}$ for the entire range $\delta < c < 2-\delta$.  

We have already mentioned the close connection between Conjectures \ref{APConj} and \ref{ShortConj} on the one hand and conjectures for the moments of Dirichlet $L$-functions and the Riemann zeta-function on the other. Indeed, our method of proof is based on the asymptotic large sieve, developed by Conrey, Iwaniec, and the second author and applied in \cite{CoIwSo} by them to prove estimates for the $6^{th}$ moment of Dirichlet $L$-functions averaged over moduli, and subsequently used by Chandee and Li \cite{ChLi} to prove estimates for the $8^{th}$ moment under the assumption of GRH. Our technique closely follows those papers.\footnote{It is worth noting in this connection that $\int_0^k \gamma_k(c)\,dc = G(1+k)^2/G(1+2k)$, the same constant conjectured to appear in moments of $L$-functions in a unitary family.}

We give a brief outline of the main idea; the problem of estimating $\Delta_k(Q;X)$ may roughly be reduced to the problem of giving an accurate estimate for the sum
\begin{align*}
&\sum_q \Phi\Big(\frac{q}{Q}\Big) \sum_{\substack{n \equiv m \Mod q \\ n\neq m}} d_k(n) d_k(m) \Psi\Big(\frac{n}{X}\Big) \Psi\Big(\frac{m}{X}\Big) \\
&= \sum_{n \neq m} d_k(n) d_k(m) \Psi\Big(\frac{n}{X}\Big) \Psi\Big(\frac{m}{X}\Big) \sum_{rq = n-m} \Phi\Big(\frac{q}{Q}\Big) \\
&= \sum_{n \neq m} d_k(n) d_k(m) \Psi\Big(\frac{n}{X}\Big) \Psi\Big(\frac{m}{X}\Big) \sum_{r|(n-m)} \Phi\Big( \frac{n-m}{rQ}\Big).
\end{align*}
The advantage of rewriting the sum in this way is that the sum over $r$ is restricted to an interval of size $(n-m)/Q \approx X/Q$, which for us will be smaller than $Q$. The condition that $r|(n-m)$ may in turn be written in terms of a sum over Dirichlet characters modulo $r$, with principal characters contributing a main term and all others contributing only to an error term that may be bound using the large sieve, or slightly more effectively by assuming GRH. This main term is then possible to estimate using classical, though elaborate, techniques. Of course, many number-theoretic details are left out of this rough description, including especially coprimality conditions that inexorably arise when making use of Dirichlet characters which nonetheless make the argument more cumbersome. 

Alongside this technique, we will require a non-trivial combinatorial argument to verify that the asymptotic formula we produce agrees with the piecewise polynomial that has been predicted.

The organization of this paper is as follows: In section \ref{sec:decomp} we decompose the variance into pieces that will be treated seperately. These are: diagonal terms, off-diagonal terms and a regular approximation, and a sum over Dirichlet characters that will be an error term. After developing analytic estimates that we use throughout the paper, we turn in sections \ref{sec:diagonal} -- \ref{sec:error} to proving an asymptotic formula for each of these pieces. Putting these pieces together in section \ref{sec:final} yields an asymptotic formula for the variance. Finally we show in section \ref{sec:comparison} that this asymptotic formula agrees with the $\gamma_k(c)$ prediction.

We make one more comment about these results: by rewriting $V_k(q;X)$ and $\Delta_k(Q;X)$ as a variance of sums
\begin{equation}
\label{char_sums}
\sum_n d_k(n) \chi(n) \Psi\Big(\frac{n}{X}\Big),
\end{equation}
over a family of characters $\chi$ and using a summation formula for such sums, it should be possible to extend Theorems \ref{main_thm} and \ref{main_thm_GRH} to the dual ranges
$$
k- (k+2)/k+\delta \leq c \leq k-\delta,\quad \textrm{and}\quad k-2+\delta \leq c \leq k-\delta,
$$
respectively. An idea of this sort is effectively used in \cite{CoIwSo} and \cite{ChLi} in the form of an approximate functional equation in order to compute moments. In those papers, characters $\chi$ were averaged over only \emph{primitive} characters, and for such characters summation formulas for \eqref{char_sums} are less complicated to write down. Here in relating the variance of \eqref{char_sums} to the quantities $V_k(q;X)$ and $\Delta_k(Q;X)$, we must average over also \emph{imprimitive} characters; this makes the application of summation formulas rather more cumbersome. These dual ranges are not treated in this paper.

\vspace{5pt}
\noindent {\bf Acknowledgments.}
We thank Adam Harper  for a  discussion which prompted us to think about this problem, and R\'egis de la Bret\`eche along with an anonymous referee for corrections. Some of the research for this paper was done while the first author was visiting Stanford University, which he thanks for its gracious hospitality.   The second author is partially supported through a grant from the National Science Foundation (NSF) and a Simons Investigator grant from the Simons Foundation.  

\bibliographystyle{plain} 

\section{Decomposing the variance}
\label{sec:decomp}

\subsection{An initial decomposition}

Since 
$$ 
V_k(q; X) = \sum_{\substack{ m \equiv n \Mod q\\ (mn,q)=1}} d_k(m)d_k(n) \Psi\Big(\frac mX \Big) \Psi\Big(\frac nX\Big) 
- \frac{1}{\phi(q)} \Big| \sum_{(n,q)=1} d_k(n) \Psi\Big(\frac nX\Big) \Big|^2, 
$$ 
we may write 
$$
\Delta_k(Q;X) = \mathcal{A}_k(Q;X) - \mathcal{B}_k(Q;X),
$$
with
\begin{equation}
\label{A_sum}
\mathcal{A}_k(Q;X):= \sum_q \Phi\Big(\frac{q}{Q}\Big) \sum_{\substack{m \equiv n \Mod q \\ (mn,q)=1}} d_k(n) d_k(m) \Psi\Big(\frac{m}{X}\Big) \Psi\Big(\frac{n}{X}\Big),
\end{equation}
and 
\begin{equation}
\label{B_sum}
\mathcal{B}_k(Q;X):= \sum_q \Phi\Big(\frac{q}{Q}\Big) \frac{1}{\phi(q)} \Big| \sum_{(n,q)=1} d_k(n) \Psi\Big(\frac{n}{X}\Big)\Big|^2.
\end{equation}
The quantity ${\mathcal B}_k(Q;X)$ is relatively easy to evaluate; it makes a large contribution when $X\gg Q$, which will be offset by a correspondingly large contribution from ${\mathcal A}_k(Q;X)$.  
To handle $\mathcal{A}_k$, we begin by breaking into pieces consisting of diagonal and off-diagonal terms:
$$
\mathcal{A}_k(Q;X) = \mathcal{D}_k(Q;X) + \mathcal{G}_k(Q;X),
$$
with
\begin{equation}
\label{D_sum}
\mathcal{D}_k(Q;X) := \sum_q \Phi\Big(\frac{q}{Q}\Big) \sum_{(n,q)=1} d_k(n)^2 \Psi\Big(\frac{n}{X}\Big)^2
\end{equation}
and
\begin{equation}
\label{G_sum}
\mathcal{G}_k(Q;X) := \sum_q \Phi\Big(\frac{q}{Q}\Big) \sum_{\substack{m \equiv n \Mod q \\ (mn,q)=1 \\ m \neq n}} d_k(m) d_k(n)\Psi\Big(\frac{n}{X}\Big) \Psi\Big(\frac{m}{X}\Big).
\end{equation}
The diagonal term will be easy enough to estimate (see section \ref{sec:diagonal}), and so long as $X = o(Q)$ only diagonal terms make an important contribution to $\Delta_k$. For $X \gg Q$ however $\mathcal{G}_k$ contributes to the main term and we deal with it using the asymptotic large sieve.

\subsection{An off-diagonal decomposition}
We break $\mathcal{G}_k$ into a main term and error term as follows.  Write 
$$
m= gM, \ \  n = gN \ \ \text{with }  \ g := (m,n). 
$$
Note that for integers $m\neq n$, we have $m \equiv n \Mod q$ and $(mn,q)=1$ if and only if $q|(M-N)$ and $(q,g)=1$. Hence \eqref{G_sum} can be rewritten as 
\begin{align*}
\mathcal{G}_k(Q,X) &= \sum_{\substack{m,n \\ m\neq n}} d_k(m)d_k(n) \Psi\Big(\frac{m}{X}\Big) \Psi\Big(\frac{n}{X}\Big) \sum_{\substack{(q, g)=1 \\ q|(M-N)}} \Phi\Big(\frac{q}{Q}\Big) \\
&= \sum_{\substack{m,n \\ m\neq n}} d_k(m)d_k(n) \Psi\Big(\frac{m}{X}\Big) \Psi\Big(\frac{n}{X}\Big) \sum_{\substack{a,\ell \\ a|g \\ a\ell|(M-N)}} \mu(a) \Phi\Big(\frac{a\ell}{Q}\Big),
\end{align*}
where in the second line we used M\"obius inversion to express the coprimality condition $(q,g)=1$. By letting $r > 0$ be such that $a\ell r = |M-N|$, we can rewrite the above as
$$
\sum_{\substack{m,n \\ m\neq n}} d_k(m)d_k(n) \Psi\Big(\frac{m}{X}\Big) \Psi\Big(\frac{n}{X}\Big) \sum_{\substack{a,r \\ a | g \\ M \equiv N \Mod {ar}}} \mu(a) \Phi\Big(\frac{|M-N|}{rQ}\Big).
$$
Since the function $\Phi$ is supported away from $0$, note that the condition $m \neq n$ is redundant. 

Expressing the congruence condition using Dirichlet characters, we obtain
\begin{align}
\label{MainAndErrorDef}
\notag \mathcal{G}_k(Q,X) =& \sum_{m,n} d_k(m) d_k(n) \Psi\Big(\frac{m}{X}\Big) \Psi\Big(\frac{n}{X}\Big) \sum_{\substack{a,r \\ a| g}} \frac{\mu(a)}{\phi(ar)} \sum_{\chi \Mod {ar}} \chi(M) \overline{\chi(N)} \Phi\Big(\frac{|M-N|}{rQ}\Big) \\
=& \mathcal{MG}_k(Q;X) + \mathcal{EG}_k(Q;X),
\end{align}
where we have split the sum into a main term with contributions coming only from principal characters
\begin{equation}
\label{MG_sum}
\mathcal{MG}_k(Q;X):= \sum_{m,n} d_k(m) d_k(n) \Psi\Big(\frac{m}{X}\Big) \Psi\Big(\frac{n}{X}\Big) \sum_{\substack{a,r \\ a |g \\ (ar, MN)=1}} \frac{\mu(a)}{\phi(ar)} \Phi\Big(\frac{|M-N|}{rQ}\Big),
\end{equation}
and (what will turn out to be) an error term with contributions from all remaining characters
\begin{equation}
\label{EG_sum}
\mathcal{EG}_k(Q;X):= \sum_{m,n} d_k(m) d_k(n) \Psi\Big(\frac{m}{X}\Big) \Psi\Big(\frac{n}{X}\Big) \sum_{\substack{a,r \\ a| g}} \frac{\mu(a)}{\phi(ar)} \sum_{\substack{\chi \Mod {ar} \\ \chi \neq \chi_0}} \chi(M) \overline{\chi(N)} \Phi\Big(\frac{|M-N|}{rQ}\Big).
\end{equation}
In order to work more easily with \eqref{MG_sum} and \eqref{EG_sum} we define for $x,y,u \geq 0$ the function
\begin{equation}
\label{W_def}
\mathcal{W}(x,y;u):= \Psi(x)\Psi(y)\Phi(u|x-y|),
\end{equation}
so that
$$
\Psi\Big(\frac{m}{X}\Big)\Psi\Big(\frac{n}{X}\Big) \Phi\Big( \frac{|M-N|}{rQ}\Big) = \mathcal{W}\Big(\frac{gM}{X},\frac{gN}{X}; \frac{X}{grQ}\Big).
$$
We will deduce some analytic properties of the function $\mathcal{W}$ below in Section \ref{sec:analytic}.

To summarize what we have shown in this section: we have the decomposition 
\begin{equation}
\label{summary}
\Delta_k = \mathcal{D}_k - \mathcal{B}_k + \underbrace{\mathcal{MG}_k + \mathcal{EG}_k}_{\mathcal{G}_k}.
\end{equation}

\section{Mellin transforms of weight functions}
\label{sec:analytic}

In this section we collect in one place some analytic estimates that we will need in the course of our proof. Since this material is somewhat technical, the reader may wish to skim through the results in this section at first and the come back to them when they are called upon. 

In what follows we frequently make use of functions analytic in multiple variables. We do not require any sophistication here: that a function $f(s,z)$ is analytic in $s$ and $z$ in a given region means for us in what follows just that for fixed $z$, $f(s,z)$ is analytic in $s$, and likewise for fixed $s$, $f(s,z)$ is analytic in $z$.


For a smooth function $G$ compactly supported in the positive reals, we denote the Mellin transform by
\begin{equation}
\label{Mellin}
\widetilde{G}(z):= \int_0^\infty G(u) u^{z-1}\,du,
\end{equation}
so that the inverse Mellin transform is given by
\begin{equation}
\label{InverseMellin}
G(u) = \frac{1}{2\pi i} \int_{(\alpha)} \widetilde{G}(z) u^{-z}\, dz,
\end{equation}
for any vertical line $\Re z = \alpha$.

\begin{proposition}
\label{Mellin_decay}
Let $G$ be a smooth function compactly supported inside $(0,\infty)$, and fix a positive constant $A$. Then uniformly for $-A \leq \Re z \leq A$,
$$
\widetilde{G}(z) \ll_\ell \frac{1}{1+|\Im z|^\ell},
$$
for all positive integers $\ell$.
\end{proposition}

\begin{proof}  Since $G$ is compactly supported inside $(0,\infty)$, it follows that 
$\widetilde{G}(z)  \int_0^{\infty} |G(z)| dz \ll 1$, so that the claimed result holds if $|\Im z| \le 1$.  If $|\Im z|\ge 1$, then 
integrating by parts $\ell$ times gives 
$$ 
\widetilde{G}(z) = (-1)^{\ell}  \int_0^\infty G^{(\ell)}(u)  \frac{u^{z+\ell-1}}{\prod_{j=1}^{\ell} (z+j-1)} du \ll_{\ell} \frac{1}{1+|\Im z|^{\ell}}.  
$$ 
\end{proof}

We will also need to make use of multivariable Mellin transforms, for the function $\mathcal{W}$ in particular. Define
\begin{equation}
\label{W2tilde}
\widetilde{\mathcal{W}}_2(s_1,s_2;u):= \int_0^\infty\int_0^\infty \mathcal{W}(x,y;u) x^{s_1-1} y^{s_2-1}\,dxdy,
\end{equation}
\begin{equation}
\label{W3tilde}
\widetilde{\mathcal{W}}_3(s_1,s_2;z):= \int_0^\infty\int_0^\infty\int_0^\infty \mathcal{W}(x,y;u) u^{z-1} x^{s_1-1} y^{s_2-1}\,du dx dy,
\end{equation}
and 
\begin{equation}
\label{wtilde}
\widetilde{w}(s_1,s_2,z) := \int_0^\infty \int_0^\infty \Psi(x) \Psi(y) \frac{x^{s_1-1} y^{s_2-1}}{|x-y|^z} \, dx dy.
\end{equation}
In these definitions, $\widetilde{\mathcal{W}}_2$ is defined for all $s_1,s_2 \in \mathbb{C}$ owing to the compact support of $\Psi$, while
$$
\widetilde{\mathcal{W}}_3(s_1,s_2;z) = \widetilde{\Phi}(z)\widetilde{w}(s_1,s_2,z),
$$
with both $\widetilde{\mathcal{W}}_3$ and $\widetilde{w}$ well-defined for all $s_1, s_2 \in \mathbb{C}$ and $\Re z < 1$.

Because the function $\mathcal{W}$ is integrable and continuous, we have the multivariable Mellin inversion formulas
$$
\mathcal{W}(x,y;u) = \frac{1}{(2\pi i)^2} \int_{(\alpha)}\int_{(\beta)} \widetilde{\mathcal{W}}_2(s_1,s_2;u) x^{-s_1} y^{-s_2}\, ds_2 ds_1,
$$
and
$$
\mathcal{W}(x,y;u) = \frac{1}{(2\pi i)^3} \int_{(\alpha)}\int_{(\beta)}\int_{(\gamma)} \widetilde{\mathcal{W}}_3(s_1,s_2;z) x^{-s_1} y^{-s_2} u^{-z}\, dz ds_2 ds_1,
$$
where $\alpha$ and $\beta$ may be freely chosen for the line of integration, while $\gamma < 1$, and the integrals over $s_1$ and $s_2$ are understood in the principal value sense.

\begin{proposition}
\label{Mellin_AC}
For $\Re z < 1$, we have 
$$
\widetilde{w}(s_1,s_2;z) = \frac{2}{1-z} \widetilde{(\Psi^2)}(s_1+s_2-1) + \mathcal{H}(s_1,s_2;z)
$$
where $\mathcal{H}(s_1,s_2;z)$ is a function that is analytic for all $s_1, s_2$ and $0< \Re z < 2.$ Moreover, for any fixed $\delta > 0$, and $s_1, s_2$ in a fixed compact region, the function $\mathcal{H}(s_1,s_2;z)$ is bounded for $0 \leq \Re z \leq 2-\delta$, and 
$\widetilde{\mathcal{W}}_3(s_1,s_2;z)$ has a meromorphic continuation to $\Re z <2 $ furnished by 
$$
\widetilde{\mathcal{W}}_3(s_1,s_2;z) = \frac{2\widetilde{\Phi}(z)}{1-z} \widetilde{(\Psi^2)}(s_1+s_2-1) + \widetilde{\Phi}(z)\mathcal{H}(s_1,s_2;z).
$$
\end{proposition}

\begin{proof}  The function $\Psi$ is smooth and compactly supported in $(0,\infty)$; say, the support of $\Psi$ is in $[u_1,u_2] \subset (0,\infty)$.  For $s_2$ restricted to a compact region, clearly 
	$$
	\Psi(y)y^{s_2-1} = \Psi(x) x^{s_2-1} + O(|x-y|). 
	$$
Put
$$
	\Omega(x,y;s_1,s_2):= \Psi(x)x^{s_1-1}(\Psi(y)y^{s_2-1} - \Psi(x)x^{s_2-1}).
$$
	Then
	\begin{equation}
	\label{decomposition1}
	\widetilde{w}(s_1,s_2; z) = \int_{u_1}^{u_2}\int_{u_1}^{u_2} \Psi(x)^2 \frac{x^{s_1+s_2-2}}{|x-y|^z}\,dxdy + \int_{u_1}^{u_2}\int_{u_1}^{u_2}\frac{\Omega(x,y; s_1, s_2)}{|x-y|^z}\,dxdy.
	\end{equation}
	For any $\epsilon > 0$,
	$$
	\mathcal{I}_\epsilon(s_1,s_2;z):= \int_{u_1}^{u_2} \int_{u_1}^{u_2} \frac{\Omega(x,y;s_1,s_2)}{|x-y|^z} \mathbf{1}_{(\epsilon,\infty)}(|x-y|)\, dx dy
	$$
is analytic for all $s_1, s_2, z$. 
Since $\Omega(x,y;s_1,s_2) = O(|x-y|)$ for $s_1$, $s_2$, $z$ restricted to any compact region with $\Re z < 2$ at all points in the region, we have that $\mathcal{I}_\epsilon(s_1, s_2; z)$ tends uniformly to
	$$
	\mathcal{I}(s_1,s_2;z):= \int_{u_1}^{u_2} \int_{u_1}^{u_2} \frac{\Omega(x,y;s_1,s_2)}{|x-y|^z}\, dx dy,
	$$
	so $\mathcal{I}(s_1,s_2;z)$ as well is analytic for all $s_1$, $s_2$, and $z$ with $\Re z < 2$. That $\mathcal{I}$ is bounded for $s_1, s_2$ in a fixed compact region with $\Re z \leq 2-\delta$ is evident. 
	
	Moreover, for $0 < \Re z < 1$,
	\begin{align}
	\label{decomposition2}
	\notag \int_{u_1}^{u_2}\int_{u_1}^{u_2} \Psi(x)^2 \frac{x^{s_1+s_2-2}}{|x-y|^z}\,dxdy =& \int_{u_1}^{u_2}\Big( \int_0^{u_2-x} \frac{dt}{t^z} + \int_0^{x-u_1} \frac{dt}{t^z}\Big) \Psi(x)^2 x^{s_1+s_2-2}\,dx \\
	\notag =& \frac{2}{1-z} \int_0^\infty \Psi(x)^2 x^{s_1+s_2-2}\,dx \\
	\notag &+ \int_{u_1}^{u_2} \Big( \frac{(u_2-x)^{1-z}-1}{1-z} + \frac{(x-u_1)^{1-z}-1}{1-z}\Big) \Psi(x)^2 x^{s_1+s_2-2}\,dx\\
	=& \frac{2}{1-z} \widetilde{(\Psi^2)}(s_1+s_2-1) + \mathcal{J}(s_1,s_2;z).
	\end{align}
Clearly $\mathcal{J}$ is analytic for all $s_1, s_2,$ and $z$, and bounded for $s_1, s_2, z$ restricted to any compact region.
	
	Combining the two decompositions \eqref{decomposition1} and \eqref{decomposition2} above, letting $\mathcal{H} = \mathcal{I}+\mathcal{J}$ gives the lemma.
\end{proof}

In the next two propositions, we let $s_1 = \sigma_1 + i t_1$ and $s_2 = \sigma_2 + i t_2$.

\begin{proposition}
\label{W3_decay}
Fix positive constants $\epsilon$ and $A$. Uniformly for $-A \leq \Re z \leq 1-\epsilon$ and $-A \leq \sigma_1, \sigma_2 < A$,
$$
\widetilde{\mathcal{W}}_3(s_1,s_2;z) \ll_\ell \frac{1}{1+|\Im z|^\ell}\,\frac{1}{1+|t_1+t_2|^\ell} \,\frac{1}{1+|t_1-t_2|^{1-\Re z}}.
$$
for all positive integers $\ell$.
\end{proposition}

\begin{proof}
We prove a closely related 
bound for the function $\widetilde{w}(s_1,s_2;z)$ first.  
Let $\Psi_1(x) = x^{-A-1}e^x \Psi(x)$, so that like $\Psi$, the function $\Psi_1$ is smooth and compactly supported, with support in $(0,\infty)$. By Mellin inversion and Fubini, for $-A \leq \Re z \leq 1-\epsilon$ and $-A \leq \sigma_1, \sigma_2 < A$, we have
\begin{equation}
\label{w_bound}
\widetilde{w}(s_1,s_2;z) = \frac{1}{(2\pi i)^2}\int_{(0)}\int_{(0)} \widetilde{\Psi_1}(\zeta_1) \widetilde{\Psi_1}(\zeta_2) E(A+1+s_1-\zeta_1,\,A+1+s_2-\zeta_2;\,z)\, d\zeta_1d\zeta_2,
\end{equation}
where
$$
E(s_1,s_2;z):= \int_0^\infty \int_0^\infty \frac{e^{-x} e^{-y} x^{s_1-1} y^{s_2-1}}{|x-y|^z}\, dxdy,
$$
is plainly well-defined and bounded for $\Re s_1, \Re s_2 \geq 1$ and $\Re z \leq 1-\epsilon$. We will bound $E(s_1,s_2;z)$, which will translate into a bound for $\widetilde{w}(s_1,s_2;z)$ owing to the rapid decay of $\widetilde{\Psi_1}$. In the definition of $E$, we make the change of variables $x = y\lambda,\; dx = y\, d\lambda$, followed by $\tau = y(1+\lambda),\; d\tau = (1+\lambda)\,dy$, to obtain 
$$
E(s_1,s_2;z) = \Gamma(s_1+s_2-z) \int_0^\infty \frac{\lambda^{s_1-1}}{(1+\lambda)^{s_1+s_2-z} |1-\lambda|^z}\, d\lambda.
$$
By splitting the integral into two pieces and using Euler's integral representation for hypergeometric functions \cite[Thm. 2.2.1]{AnAsRo}, we see that 
\begin{align*}
\int_0^\infty \frac{\lambda^{s_1-1}}{(1+\lambda)^{s_1+s_2-z} |1-\lambda|^z}\, d\lambda =& \;\int_0^1 \frac{\lambda^{s_1-1}}{(1+\lambda)^{s_1+s_2-z}(1-\lambda)^z}\, d\lambda \\
&+ \int_1^\infty \frac{\lambda^{s_1-1}}{(1+\lambda)^{s_1+s_2-z}(\lambda-1)^z}\, d\lambda\\
=& \;\frac{\Gamma(1-z)\Gamma(s_1)}{\Gamma(1+s_1-z)} \,\setlength\arraycolsep{1pt}
\,_2 F_1\Big(\begin{matrix}s_1 &,\; s_1+s_2-z\\& 1+ s_1-z\quad\end{matrix};-1\Big) \\
&+ \frac{\Gamma(1-z)\Gamma(s_2)}{\Gamma(1+s_2-z)} \,\setlength\arraycolsep{1pt}
\,_2 F_1\Big(\begin{matrix}s_2 &,\; s_1+s_2-z\\& 1+ s_2-z\quad\end{matrix};-1\Big).
\end{align*}
We have made a change of variables $\lambda = 1/\ell$ in the second integral in order to simplify it.

On the other hand, by Barnes' integral for the hypergeometric function \cite[Thm 2.4.1]{AnAsRo}, for $1\leq \Re a, \Re b, \Re c \leq B$ for a fixed constant $B$,
$$
\setlength\arraycolsep{1pt}
\,_2 F_1\Big(\begin{matrix}a &,\; b\\& c\quad\end{matrix};-1\Big) = \frac{1}{2\pi i} \int_{-i \infty}^{i\infty} \frac{\Gamma(a+s)}{\Gamma(a)} \frac{\Gamma(b+s)}{\Gamma(b)}\frac{\Gamma(c)}{\Gamma(c+s)} \Gamma(-s)\,ds,
$$
where the path of integration is a straight line except for a small (radius $1/2$ say) semi-circle around the left of the origin in order to miss the singularity of $\Gamma(-s)$. By Stirling's formula \cite[Thm C.1]{MoVa} we see this is bounded for $1\leq \Re a, \Re b, \Re c \leq B$. Hence, for $1 \leq \Re s_1, \Re s_2 \leq B$ and $-A \leq \Re z \leq 1-\epsilon$,
$$
E(s_1,s_2;z) \ll \Gamma(s_1+s_2-z)\Big(\frac{\Gamma(1-z)\Gamma(s_1)}{\Gamma(1+s_1-z)}+ \frac{\Gamma(1-z)\Gamma(s_2)}{\Gamma(1+s_2-z)}\Big).
$$
Making use of Stirling's formula, this is
\begin{multline*}
\ll_\ell \frac{1}{1+|t_1+t_2-z|^\ell} \Big( \frac{(|t_1|+1)^{\sigma_1-1/2}(1+|\Im z|)^{1-\Re z -1/2}}{(1+|t_1-\Im z|)^{\sigma_1 + 1 - \Re z -1/2}} \\+ \frac{(|t_2|+1)^{\sigma_2-1/2}(1+|\Im z|)^{1-\Re z -1/2}}{(1+|t_2-\Im z|)^{\sigma_2 + 1 - \Re z -1/2}}\Big).
\end{multline*}

Hence, for $s_1, s_2, z$ as in the proposition, applying this in \eqref{w_bound} gives us the same bound for $w(s_1,s_2;z)$, with $\sigma_1$ and $\sigma_2$ replaced by $\sigma_1+A+1$ and $\sigma_2+A+1$ respectively.

Finally, because 
$$
\widetilde{\mathcal{W}}_3(s_1,s_2;z) = \widetilde{\Phi}(z)\widetilde{w}(s_1,s_2;z),
$$
and because $\widetilde{\Phi}(z) \ll_\ell 1/(1+|\Im z|^\ell)$, this bound for $\widetilde{w}$ gives us
$$
\widetilde{\mathcal{W}}_3(s_1,s_2;z) \ll_\ell \frac{1}{1+|\Im z|^\ell}\,\frac{1}{1+|t_1+t_2|^\ell} \Big( \frac{1}{1+|t_1|^{1-\Re z}} + \frac{1}{1+|t_2|^{1-\Re z}}\Big),
$$
which is equivalent to the claimed bound.
\end{proof}

\begin{proposition}
\label{W2_bound}
Fix $A>0$. Uniformly for $-A \leq \sigma_1, \sigma_2 \leq A$ and $u \geq 0$, we have 
$$
\widetilde{\mathcal{W}}_2(s_1,s_2;u) \ll_p \frac{u^{p-1}}{1+\max(|t_1|,|t_2|)^p}, 
$$ 
for all $p \geq 1$.  Further, there is a constant $K>0$ such that $\widetilde{\mathcal W}_2(s_1,s_2;u)=0$ unless $u\ge K$.  
\end{proposition}

\begin{proof}
Integrating by parts $p$ times furnishes the bound on $\widetilde{\mathcal W}_2(s_1,s_2;u)$; the exponent of $u$ is $p-1$ because 
of the $\Phi(u|x-y|)$ term in the definition of ${\mathcal W}$ which forces $|x-y|$ to be on the scale of $1/u$.  
If $u$ is sufficiently small, then ${\mathcal W}(x,y;u)$ vanishes for all $x$ and $y$ -- this being a consequence of the support of $\Phi$ and $\Psi$. 
\end{proof}





\section{The diagonal contribution, and $\Delta_k$ for $X = o(Q)$}
\label{sec:diagonal}

In this section we estimate the diagonal sum $\mathcal{D}_k(Q;X)$, obtaining a good estimate for all ranges of $Q$ and $X$. In the easy range $X\leq \eta Q$, for a certain constant $\eta$, this gives an asymptotic formula for $\Delta_k(Q;X)$.

\subsection{Fixed moduli}

By a standard contour shift argument, we can estimate the diagonal contributions for individual $q$.  

\begin{proposition}
\label{Diagonal_fixedmoduli}
Fix a constant $C$. For all $\epsilon > 0$,
$$
\sum_{(n,q)=1} d_k(n)^2 \Psi\Big(\frac{n}{X}\Big)^2 = \frac{a_k(q)}{(k^2-1)!} X (\log X)^{k^2-1} + O_\epsilon(X(\log X)^{k^2-2+\epsilon}),
$$
uniformly for $q \leq X^C$.
\end{proposition}

The proof gives an asymptotic formula of the form $X P(q;\log X) + O_\epsilon( X^{1/2+\epsilon}q^\epsilon)$ uniformly for all $q$, where $P(q;\cdot)$ is a $k^2-1$ degree polynomial with (somewhat complicated) coefficients that depend on $q$.

\begin{proof}
Mellin inversion gives for any $\alpha > 1$ 
\begin{equation}
\label{Mellin_diagonal}
\sum_{(n,q)=1} d_k(n)^2 \Psi\Big(\frac{n}{X}\Big)^2 = \frac{1}{2\pi i} \int_{(\alpha)} \widetilde{(\Psi^2)}(s) X^s \sum_{(n,q)=1} \frac{d_k(n)^2}{n^s}\,ds.
\end{equation}
For $\Re s > 1$ we may write 
$$
\sum_{(n,q)=1} \frac{d_k(n)^2}{n^s} = \prod_{p \nmid q} \sum_{\ell=0}^\infty \binom{\ell+k-1}{k-1}^2 p^{-\ell s} = \zeta(s)^{k^2}  F(s) f_q(s),
$$
where 
$$
f_q(s):= \prod_{p \mid q} \Big( \sum_{\ell =0}^{\infty} \binom{\ell + k-1}{k-1}^2 p^{-\ell s} \Big)^{-1},
$$
and
$$
F(s):= \prod_{p} \Big(1-\frac{1}{p^s}\Big)^{k^2} \sum_{\ell=0}^\infty \binom{\ell+k-1}{k-1}^2 p^{-\ell s}.
$$
The Euler product defining $F(s)$ converges absolutely in $\Re s >1/2$, and so in the region $\Re s \ge 1/2+\epsilon$ 
we have $F(s) \ll_{\epsilon} 1$.  The product defining $f_q(s)$ converges when $\Re s >0$, and in the 
region $\Re s \ge 1/2$ we have   
$$
f_q(s) \ll \prod_{p \mid q} \Big(1 + O\Big(\frac{1}{p^{1/2}}\Big)\Big)  \ll_\epsilon q^\epsilon. 
$$

Shifting the contour in \eqref{Mellin_diagonal} from $\alpha$ to $1/2+\epsilon$ (and noting that $|\zeta(s)|$ grows only polynomially in 
$|s|$ for $\Re s\ge 1/2$; see \cite[Ch. V]{Ti}) ,   we thus have
\begin{align}
\label{Res1}
\notag \sum_{(n,q)=1} d_k(n)^2 \Psi\Big(\frac{n}{X}\Big)^2 &=  \Res_{s=1}\;\widetilde{(\Psi^2)}(s) X^s \zeta(s)^{k^2} f_q(s) F(s)  \\
\notag &\quad+ \frac{1}{2\pi i} \int_{(1/2+\epsilon)} \widetilde{(\Psi^2)}(s) X^s \zeta(s)^{k^2} f_q(s) F(s)\,ds \\
&= \Res_{s=1}\;\widetilde{(\Psi^2)}(s) X^s \zeta(s)^{k^2} f_q(s) F(s)+ O_\epsilon( X^{1/2+\epsilon} q^\epsilon).
\end{align}
Expanding $f_q(s)$ into its Taylor series around $1$, the residue above may be written as 
\begin{equation} 
\label{Res2}
\sum_{j=0}^{k^2-1} \frac{f_q^{(j)}(1) }{j!} R_{k^2-1-j}(\log X), 
\end{equation} 
where $R_{k^2-1-j}$ is a polynomial of degree $k^2-1-j$ with coefficients determined by the Laurent expansion of 
$\widetilde{(\Psi^2)}\zeta(s)^{k^2} F(s)$ (and thus independent of $q$).  

For a prime $p$ consider the Euler factor in the definition of $f_q(s)$, call it temporarily $e_p(s)$.  This may be expanded into 
a power series around $1$:   
$$
e_p(s) = e_p(1) \big( 1+ b_1(p) (s-1) + b_2(p) (s-1)^2 +\ldots \big), 
$$ 
with $b_j(p) \ll_j (\log p)^j/p$.  Multiplying this expansion over all $p|q$ we find 
$$ 
f_q(s) = f_q(1) \big( 1+ c_1(q) (s-1) + c_2(q) (s-1)^2 + \ldots \big), 
$$ 
where we may see that $c_j(q) \ll_j (\log \log q)^j$.  The bound on $c_j(q)$ follows from the bound on $b_j(p)$ 
together with the bound $\sum_{p|q} (\log p)^j/p \ll (\log \log q)^{j}$ (attained for primorials $q$).   Using these 
observations in \eqref{Res2}, we see that the quantity in \eqref{Res2} is 
\begin{align*}
&f_q(1) R_{k^2 -1} (\log X) + O((\log X)^{k^2-2} \log \log q) \\
=& f_q(1)  \Big( \widetilde{(\Psi^2)}(1) F(1) \frac{(\log X)^{k^2-1}}{(k^2-1)!} + O((\log X)^{k^2-2}) \Big)
+ O((\log X)^{k^2-2+\epsilon}). 
\end{align*} 
Noting that $\widetilde{(\Psi^2)}(1)=1$ and that $f_q(1)F(1) =a_k(q)$, the proposition follows. 
 \end{proof}

\subsection{Averaged moduli}

By using standard contour integration techniques as above one may see that
\begin{proposition}
\label{averaging_constants}
For any $\epsilon > 0$,
$$
\sum_q a_k(q) \Phi\Big(\frac{q}{Q}\Big) = \widetilde{a}_k Q + O_\epsilon(Q^{1/2+\epsilon}).
$$
\end{proposition}
We leave details to the reader. 

By combining Propositions \ref{Diagonal_fixedmoduli} and \ref{averaging_constants} we obtain the diagonal piece of the asymptotic formula \eqref{smooth_variance2} in Theorem \ref{main_thm}.
\begin{lemma}[Diagonal terms]
\label{diagonal}
Fix $\delta > 0$. For $X \geq Q^\delta$ and for any $\epsilon > 0$,
$$
\mathcal{D}_k(Q;X) = \frac{\widetilde{a}_k}{(k^2-1)!} QX (\log X)^{k^2-1} + O_\epsilon(Q^{1/2+\epsilon} X (\log X)^{k^2-1} + QX (\log X)^{k^2-2+\epsilon}).
$$
Hence for $c:=\frac{\log X}{\log Q}$ uniformly for $c \geq \delta$,
$$
\mathcal{D}_k(Q;X) \sim \widetilde{a}_k \frac{c^{k^2-1}}{(k^2-1)!} QX(\log Q)^{k^2-1}.
$$
\end{lemma}

\subsection{Estimating $\Delta_k$ for small $X$}

In the range $X=o(Q)$, the condition $m\equiv n \Mod q$ forces $m=n$.  Thus, in this range ${\mathcal G}_k(Q;X)=0$. 
Moreover it is straightforward to see that (when $X=o(Q)$ and $k\ge 2$)  
$$ 
{\mathcal B}_k(Q;X) \ll \big( X (\log X)^{k-1} \big)^2 \log Q \ll X^2 (\log Q)^{2k-1}  = o(XQ (\log Q)^{k^2-1}).
$$ 
Using our evaluation of ${\mathcal D}_k$, we conclude the following. (To fix our imagination, rather than just $X = o(Q)$, we take $X \leq Q/\log \log Q$.) 

\begin{lemma}[Theorem \ref{main_thm} for small $c$]
\label{small_c}
Fix $\delta > 0$. For $Q^\delta \leq X \leq Q/\log \log Q$, with $c:=\frac{\log X}{\log Q}$ we have 
$$
\Delta_k(Q;X) \sim \widetilde{a}_k \frac{c^{k^2-1}}{(k^2-1)!} QX (\log Q)^{k^2-1}.
$$
\end{lemma}
We will see later (see Section 7) by explicitly calculating $\gamma_k(c)$ that this establishes Theorem \ref{main_thm} for $c \in (0,1)$. 

Since we have proved an estimate for diagonal sums that is uniform in $q$, this may also be used in exactly the same way to establish a smoothed version of Conjecture \ref{APConj} for $c\in(0,1)$. On the other hand, to go past this diagonal analysis, we will really require the averaging in $q$.

\section{Off-diagonal asymptotics}
\label{sec:offdiagonal}

We turn now to the terms $\mathcal{MG}_k-\mathcal{B}_k$ in the decomposition \eqref{summary}. For $X = o(Q)$ we have just shown that these are error terms, but for $X \gg Q$ they instead make a contribution to the leading order of $\Delta_k$. We will prove both an unconditional and a conditional asymptotic formula for $X \geq Q/\log \log Q$.

Let $P_k(c)$ be the polynomial
\begin{equation}
\label{P_k}
P_k(c):= \Res_{z=0} \Res_{\substack{s_1=0 \\ s_2=0}} e^z e^{c(s_1+s_2-z)} \frac{(s_1-z)^k (s_2-z)^k}{z^2 s_1^k s_2^k (s_1+s_2-z)^{k^2}}.
\end{equation}
\begin{remark}
In Proposition \ref{P_k_expand}, we will record a slightly more traditional representation for the polynomial $P_k(c)$, involving a (complicated) sum of binomial coefficients.
\end{remark}

Our sole purpose in the rest of this section is to demonstrate the following:

\begin{lemma} 
\label{off_diagonal_unconditional}
Fix $\delta > 0$. For $X$ in the range $Q/\log \log Q \leq X \leq Q^{(k+2)/k-\delta}$
\begin{equation}
\label{MG-B_asymptote}
\mathcal{MG}_k(Q;X) - \mathcal{B}_k(Q;X) = \widetilde{a}_k P_k\Big(\frac{\log X}{\log Q}\Big) XQ (\log Q)^{k^2-1} + O_\epsilon(XQ (\log X)^{k^2-2+\epsilon}),
\end{equation}
for any $\epsilon > 0$.

Moreover, on the Riemann hypothesis \eqref{MG-B_asymptote} is true in the larger range $Q/\log \log Q \leq X \leq Q^{2-\delta}$.
\end{lemma}

\begin{proof}
Because this proof is somewhat lengthy, we break it up into three steps.

\textbf{Step 1: Approximating $\mathcal{MG}_k$ by (two) contour integrals.}
We will make use of Mellin transforms to approximate $\mathcal{MG}_k(Q;X)$, pulling off one term that will ultimately match up with $\mathcal{B}_k(Q;X)$ and another that will give rise to the main term in Lemma \ref{off_diagonal_unconditional}.

Note that
$$
\mathcal{MG}_k(Q;X) = \sum_{m,n} \sum_{\substack{a,r \\ a|g \\ (ar, MN)=1}} d_k(m) d_k(n) \frac{\mu(a)}{\phi(ar)} \mathcal{W}\Big(\frac{m}{X},\frac{n}{X}; \frac{X}{grQ}\Big).
$$
Defining
$$
T_k(s_1,s_2;z):= \sum_{m,n} \sum_{\substack{a,r \\ a|g \\ (ar, MN)=1}} \frac{d_k(m)}{m^{s_1}} \frac{d_k(n)}{n^{s_2}} \frac{\mu(a)}{\phi(ar)} (gr)^z,
$$
for $\Re s_1, \Re s_2 > 1$ and $\Re z < 0$, we have for any $A > 0$ (a parameter which will be chosen more exactly later) and small $\epsilon > 0$
\begin{equation}
\label{MG_Mellin}
\mathcal{MG}_k(Q:X) = \frac{1}{(2\pi i)^3} \int_{(-A)} \int_{(1+\epsilon)} \int_{(1+\epsilon)} \widetilde{\mathcal{W}}_3(s_1,s_2;z) X^{s_1} X^{s_2} (Q/X)^z \,T_k(s_1, s_2, z)\, ds_1 ds_2 dz.
\end{equation}
(Since the three variable integral is absolutely convergent, one needn't worry at this point about the order of the contours over $s_1, s_2,$ and $z$.)

For $\Re s_1, \Re s_2 > 1$ and $\Re z < 0$, note by multiplicativity that
$$
T_k(s_1, s_2; z) = \prod_p \sum_{\mu, \nu, \alpha, \rho} \frac{d_k(p^\mu) d_k(p^\nu)}{p^{\mu s_1} p^{\nu s_2}} \frac{\mu(p^\alpha)}{\phi(p^{\alpha+\rho})} (p^{\min(\mu,\nu)+\rho})^z,
$$
where the sum is over all tuples of non-negative integers $(\mu,\nu,\alpha,\rho)$ satisfying
$$
\alpha \leq \min(\mu,\nu), \quad \textrm{and either}\; \alpha = \rho = 0,\;\textrm{or}\; \mu = \nu.
$$
This is a complicated expression, but we note that by an inspection of the Euler product,
\begin{equation}
\label{T_factor}
T_k(s_1,s_2;z) = \zeta(1-z)\zeta(2-z) \frac{\zeta(s_1)^k \zeta(s_2)^k \zeta(s_1+s_2-z)^{k^2}}{\zeta(1+s_1-z)^k \zeta(1+s_2-z)^k} V_k(s_1,s_2;z),
\end{equation}
where $V_k(s_1,s_2;z)$ is a function that is analytic and bounded in the regions
$$
\Re s_1, \Re s_2 \geq 1/2 +c,\quad \Re(s_2+s_2-z) \geq 1/2 + c, \quad \Re z \leq 1+c,
$$
for any $c$ chosen with $0 < c < 1/2$.

We return to the expression \eqref{MG_Mellin} and shift the contours of $s_1$ and $s_2$ each from the line $(1+\epsilon)$ to $\Gamma$, where $\Gamma$ is a contour running in straight line-segments from

$$
(1+\epsilon)-i\infty \;\; \textrm{to} \;\; (1+\epsilon) - iY  \;\; \textrm{to} \;\; (1/2+\epsilon) - iY  \;\; \textrm{to} \;\; (1/2+\epsilon) + iY  \;\; \textrm{to} \;\; (1+\epsilon)+iY  \;\; \textrm{to} \;\; (1+\epsilon)+i\infty,
$$
\vspace{1pt}
where $Y$ is a (large) parameter to be chosen later. Shifting the contour of $s_1$ first and then $s_2$, we see
\begin{multline}
\label{Mellin_eval_s1s2}
\int_{(1+\epsilon)}\int_{(1+\epsilon)} \widetilde{\mathcal{W}}_3(s_1,s_2;z)  X^{s_1+s_2-z}\, T_k(s_1,s_2;z)\,ds_1 ds_2 \\=  \Res_{s_1,s_2=1} \Big(\widetilde{\mathcal{W}}_3(s_1,s_2;z)  X^{s_1+s_2-z}\, T_k(s_1,s_2;z)\Big) + ErrorTerm,
\end{multline}
where
\begin{align*}
ErrorTerm &=\Big( \int_{(1+\epsilon)} \int_{\Gamma} + \int_\Gamma \int_{(1+\epsilon)} - \int_\Gamma\int_\Gamma  \Big)\widetilde{\mathcal{W}}_3(s_1,s_2;z)  X^{s_1+s_2-z}\, T_k(s_1,s_2;z)\, ds_1ds_2 . 
\end{align*} 
Applying Proposition \ref{W3_decay} to bound $\widetilde{\mathcal{W}}_3$ (with $\ell$ being an integer larger than $1+A$), 
and using \eqref{T_factor} and that $\Re z = -A$, we 
may bound the integrand above by 
$$ 
\ll  \frac{1}{1+|z|^{1+A}} \frac{|\zeta(s_1)|^k |\zeta(s_2)|^k} {(1+|t_1+t_2|^{1+A}) (1+|t_1-t_2|^{1+A}) } X^{\Re(s_1+s_2)+A}. 
$$ 
If either $s_1$ or $s_2$ has real part $1+\epsilon$, then the corresponding factor of $|\zeta(s_j)|^k$ is $\ll 1$, and 
we may integrate out this variable.  If neither $s_1$ nor $s_2$ has real part $1+\epsilon$ then bound $X^{\Re(s_1+s_2)}|\zeta(s_1)\zeta(s_2)|^k$ 
by $X^{2\Re s_1}|\zeta(s_1)|^{2k} +X^{2\Re s_2} |\zeta(s_2)|^{2k}$ and then integrate out the variable not involving a power of $\zeta$.  In this way 
we obtain 
\begin{equation} 
\label{ET1}
ErrorTerm \ll \frac{1}{1+|z|^{1+A}} \int_{\Gamma} \frac{ X^{\Re s + 1 + \epsilon + A} |\zeta(s)|^{k} +X^{2\Re s +A} |\zeta(s)|^{2k}}{1+|s|^{A+1}} |ds|. 
\end{equation} 

In the part of $\Gamma$ with $|\Im s| > Y$ (so $\Re s =1+\epsilon$) we use $|\zeta(s)|\ll 1$ and see that the contribution of this 
part to the integral in \eqref{ET1} is $\ll X^{2+2\epsilon+A}/Y^A$.  To estimate the horizontal lines in $\Gamma$, we use the convexity bound 
$|\zeta(s)| \ll (1+|s|)^{(1-\Re s)/2+\epsilon}$, and obtain that these line segments contribute $\ll (X^{2+2\epsilon +A} Y^{k\epsilon} + X^{3/2+2\epsilon +A}Y^{k/4+\epsilon} + X^{1+2\epsilon+A}Y^{k/2+\epsilon})
/Y^{A+1}$.  Lastly, to bound the integrals on the line segment from $1/2+\epsilon -iY$ to $1/2+\epsilon +iY$, we split the integral into dyadic blocks and note that (for any $r>0$)
$$ 
\int_{T}^{2T} |\zeta(1/2+\epsilon +it)|^{r} dt \ll T^{1+\epsilon} + T^{r/4+\epsilon}, 
$$ 
which follows from the convexity bound for $\zeta(s)$ together with the fourth moment.   It follows that for $A > \epsilon$ the contribution of this line segment to \eqref{ET1} is 
$$ 
\ll X^{3/2+2\epsilon+A} (1 + Y^{k/4-1-A +\epsilon}) + X^{1+2\epsilon +A} (1+ Y^{k/2-1-A +\epsilon}). 
$$ 
Putting all these estimates together we conclude that 
\begin{equation} 
\label{ET2} 
ErrorTerm \ll \frac{1}{1+|z|^{1+A}} \Big( \frac{X^{2+2\epsilon+A}Y^{k\epsilon}}{Y^A} + \frac{X^{3/2+2\epsilon+A} Y^{k/4+\epsilon}}{Y^{1+A}} + 
\frac{X^{1+2\epsilon+A} Y^{k/2+\epsilon}}{Y^{1+A}} +X^{3/2+2\epsilon+A} \Big). 
\end{equation} 
If the Riemann hypothesis is assumed then $|\zeta(s)| \ll (1+|s|)^{\epsilon}$ when $\Re (s)\ge 1/2$, and we obtain the better bound 
\begin{equation} 
\label{ET3} 
ErrorTerm \ll \frac{1}{1+|z|^{1+A}} \Big( \frac{X^{2+2\epsilon+A}}{Y^A} + X^{3/2+2\epsilon+A}\Big) . 
\end{equation}

When $k= 2$, we let $Y\to \infty$ and take any $A > \epsilon$ in \eqref{ET2}, and conclude that 
$$
ErrorTerm \ll X^{3/2+2\epsilon+A}/(1+|z|^{1+A}). 
$$ 
 For $k>2$, we take $Y=X^{2/(k-2)}$ and $A= (k-2)/4$ to conclude again the same bound for $ErrorTerm$ (with $2\epsilon$ replaced by $2\epsilon+2\epsilon k/(k-2)$, a quantity still always smaller than $5\epsilon$).  
%
%
%
Returning to \eqref{Mellin_eval_s1s2}, we have (with the above choice for $A$ when $k>2$) 
\begin{multline*}
\mathcal{MG}_k(Q;X) = \frac{1}{2\pi i} \int_{(-A)} Q^z \Res_{s_1,s_2=1} \widetilde{\mathcal{W}}(s_1,s_2;z)  X^{s_1+s_2-z}\, T_k(s_1,s_2;z)\,dz \\+ O_\epsilon(Q^{-A}X^{3/2+5\epsilon+A}).
\end{multline*}
For $k= 2$ we take $A=2\epsilon$, and see that the error term above is $O(XQ)$ provided $X\le Q^{2-\delta}$ and $\epsilon$ is sufficiently small.  For $k>2$ the error term above is 
readily seen to be $O(XQ)$ provided $X\le Q^{(k+2)/2-\delta}$ and $\epsilon$ is sufficiently small.  
Hence unconditionally for $X \leq Q^{(k+2)/k -\delta}$,
\begin{equation}
\label{Mellin_eval_s1s2.2}
\mathcal{MG}_k(Q;X) = \frac{1}{2\pi i} \int_{(-A)} Q^z \Res_{s_1,s_2=1} \widetilde{\mathcal{W}}(s_1,s_2;z)  X^{s_1+s_2-z}\, T_k(s_1,s_2;z)\,dz + O(XQ).
\end{equation}
On the Riemann hypothesis, using \eqref{ET3} we obtain the same result in the wider range $X\le Q^{2-\delta}$. 

We further simplify \eqref{Mellin_eval_s1s2.2}. Writing each residue as an integral over a small contour centered at $1$:
$$
\mathcal{O}:= \{s: |s-1|=\epsilon\}
$$
we may exchange the order of integration to rewrite the integral in \eqref{Mellin_eval_s1s2.2}:
$$
\frac{1}{(2\pi i)^3} \int_\mathcal{O} \int_\mathcal{O} \int_{(-A)}  \widetilde{\mathcal{W}}_3(s_1,s_2;z) Q^z X^{s_1+s_2-z}\, T_k(s_1,s_2;z)\, dz\, ds_1 ds_2.
$$
We will shift the integral in $z$ from the line $(-A)$ to the line $(\epsilon)$. Using the estimate for $\widetilde{\mathcal{W}}_3$ of Proposition \ref{W3_decay} to bound the horizontal components of this shift, and collecting the sole residue at $z=0$ we see that
\begin{align}
\label{MG_TwoMellins}
\notag \mathcal{MG}_k(Q;X) =& \frac{1}{(2\pi i)^3} \int_\mathcal{O} \int_\mathcal{O} \int_{(\epsilon)}  \widetilde{\mathcal{W}}_3(s_1,s_2;z) Q^z X^{s_1+s_2-z}\, T_k(s_1,s_2;z)\, dz ds_1 ds_2 \\
&+ \frac{1}{(2\pi i)^2} \int_\mathcal{O} \int_\mathcal{O} - \Res_{z=0}\,  \widetilde{\mathcal{W}}_3(s_1,s_2;z)  Q^z X^{s_1+s_2-z}\, T_k(s_1,s_2;z) \, ds_1 ds_2 \\
\notag &+ O(XQ).
\end{align}
Unconditionally we have this for $X \leq Q^{(k+2)/k-\delta}$, while on the Riemann hypothesis it is true in the larger region $X \leq Q^{2-\delta}$.

We conclude by noting a simplification of the residue in \eqref{MG_TwoMellins} that we use later. From \eqref{T_factor} and the definition of $\widetilde{\mathcal{W}}_3$, we have for $s_1, s_2 \in \mathcal{O}$,
\begin{equation}
\label{Res_eval1}
- \Res_{z=0}\,  \widetilde{\mathcal{W}}_3(s_1,s_2;z)  Q^z X^{s_1+s_2-z}\, T_k(s_1,s_2;z) = X^{s_1+s_2} \widetilde{\Psi}(s_1) \widetilde{\Psi}(s_2) \widetilde{\Phi}(0) \lim_{z\rightarrow 0} - z T_k(s_1,s_2,z).
\end{equation}

We turn now to the second step of our proof to see that much the same expression occurs in an evaluation of $\mathcal{B}_k$. After that, in the third step, we will evaluate the first contour integral in \eqref{MG_TwoMellins}.

\vspace{3mm}
\textbf{Step 2: Approximating $\mathcal{B}_k$ by a matching contour integral.}
Recall
$$
\mathcal{B}_k(Q;X) = \sum_{m,n} d_k(m)d_k(n) \Psi\Big(\frac{m}{X}\Big)\Psi\Big(\frac{n}{X}\Big) \sum_{(q,mn)=1} \frac{1}{\phi(q)} \Phi\Big(\frac{q}{Q}\Big).
$$
We estimate this sum with the help of the following proposition.
\begin{proposition}
\label{phi_sum}
Let
$$
S_H(z):= \sum_{(q,H)=1} \frac{1}{\phi(q) q^z}.
$$
Then
$$
\sum_{(q,H)=1} \frac{1}{\phi(q)} \Phi\Big(\frac{q}{Q}\Big) = \widetilde{\Phi}(0) \lim_{z\rightarrow 0^+}( z S_H(z)) + O_\epsilon(H^\epsilon/Q^{1-\epsilon}),
$$
uniformly in $H$, for any $\epsilon > 0$.
\end{proposition}

\begin{proof}
Define
$$
\mathfrak{s}_p(z):= 1+ \frac{1}{(p^{1+z}-1)(1-p^{-1})}.
$$
Note that
$$
S_H(z) = \prod_{p\nmid H} \mathfrak{s}_p(z) = \zeta(1+z) \Big(\prod_{p|H}\mathfrak{s}_p(z)^{-1}\Big) R(z),
$$
where
$$
R(z) = \prod_p \Big(1-\frac{1}{p^{1+z}}\Big) \mathfrak{s}_p(z)
$$
is bounded and analytic for $\Re z \geq -1 + \delta$, for any $\delta > 0$. We have
$$
\sum_{(q,H)=1} \frac{1}{\phi(q)} \Phi\Big(\frac{q}{Q}\Big) = \frac{1}{2\pi i} \int_{(\epsilon)} \widetilde{\Phi}(z) Q^z S_H(z)\, dz,
$$
for $\epsilon > 0$. The proposition then follows by pushing the line of integration to $(-1+\epsilon)$ and noting that for $\Re z = -1 + \epsilon$,
$$
\prod_{p|H} \mathfrak{s}_p(z)^{-1} \ll_\epsilon H^\epsilon,
$$
by using the fact that $\omega(H) \ll \log H/\log \log H.$
\end{proof}

\begin{remark}
It is easy to simplify $\lim_{z\rightarrow 0^+}( z S_H(z))$ to an Euler product depending on $H$. However for us this representation as a limit will be more convenient.
\end{remark}

Applying Proposition \ref{phi_sum} to $\mathcal{B}_k$,
$$
\mathcal{B}_k(Q;X) = \lim_{z\rightarrow 0 +} \sum_{m,n} d_k(m) d_k(n) \Psi\Big(\frac{m}{X}\Big) \Psi\Big(\frac{n}{X}\Big) \widetilde{\Phi}(0) z S_{mn}(z) + O_\epsilon(X^{2+\epsilon}/Q^{1-\epsilon}).
$$
Making use of a Mellin transform to rewrite the sum over $m$ and $n$, this is
\begin{equation}
\label{B_Mellin}
\lim_{z\rightarrow 0^+} \int_{(1+\epsilon)} \int_{(1+\epsilon)} \widetilde{\Psi}(s_1)\widetilde{\Psi}(s_2)\widetilde{\Phi}(0) X^{s_1+s_2} z F_k(s_1,s_2;z)\; ds_1 ds_2 + O_\epsilon (X^{2+\epsilon}/Q^{1-\epsilon}),
\end{equation}
where
$$
F_k(s_1,s_2;z):= \sum_{\substack{m,n,q \\ (q,mn)=1}} \frac{d_k(m) d_k(n)}{m^{s_1} n^{s_2}} \frac{1}{\phi(q) q^z}.
$$
By factoring $F_k$ as an Euler product, one may check that
\begin{equation}
\label{F_factor}
F_k(s_1,s_2;z)  = \zeta(s_1)^k \zeta(s_2)^k \zeta(1+z) G_k(s_1,s_2;z)
\end{equation}
for a function $G_k(s_1,s_2;z)$ that is analytic and bounded for $\Re s_1, \Re s_2 \geq 1/2+\delta$ and $\Re z \geq -1/2 + \delta$ for any $\delta > 0$.

Returning to \eqref{B_Mellin}, we shift the contour in $s_1$ and $s_2$ from $(1+\epsilon)$ to $(1/2+\epsilon)$ each, picking up as before a residue at $s_1,s_2=1$. We make use of the rapid decay of $\widetilde{\Psi}(s_1)\widetilde{\Psi}(s_2)$ to bound the new contour. Such residues as we pick up by shifting the contour occur at $s_1,s_2=1$. Leaving these residues as contour integrals localized around $1$, we obtain:
\begin{multline}
\label{B_Mellin2}
\mathcal{B}_k(Q;X) = \lim_{z\rightarrow 0} \int_{\mathcal{O}}\int_{\mathcal{O}} \widetilde{\Phi}(0) \widetilde{\Psi}(s_1) \widetilde{\Psi}(s_2) X^{s_1+s_2} z F_k(s_1,s_2;z) \, ds_1 ds_2 \\ + \underbrace{O_\epsilon(X^{3/2+\epsilon}) + O_\epsilon(X^{2+\epsilon}/Q^{1-\epsilon})}_{= O(XQ)\quad \textrm{for}\; X\leq Q^{2-\delta}},
\end{multline}
where $\mathcal{O}$ is as before the contour $\{s: |s-1| = \epsilon\}$. Using \eqref{F_factor} to establish dominated convergence, we may transfer the limit inside the integrals. 

We claim that the integral that occurs here is the same as the second integral that appeared in \eqref{MG_TwoMellins}. By using \eqref{Res_eval1}, to demonstrate this we need only show that 
\begin{equation}
\label{singularity_comparison}
\lim_{z\rightarrow 0} -z T_k(s_1,s_2;z) = \lim_{z\rightarrow 0} z F_k(s_1,s_2;z)
\end{equation}
for $s_1, s_2 \in \mathcal{O}$. One method for verifying this is a straightforward though tedious computation with the Euler product factorizations in \eqref{T_factor} and \eqref{F_factor}. Alternatively, note for $\Re s_1, \Re s_2 > 1$ and $\Re z < 0$,
\begin{align}
\label{T_to_G}
\notag T_k(s_1,s_2;z) &= \sum_{\substack{m,n,q \\ (q,mn)=1}} \frac{d_k(m) d_k(n)}{m^{s_1} n^{s_2}} \frac{q^z}{\phi(q)} \sum_{a | (g,q)} (g/a)^z \mu(a) \\
&= F_k(s_1,s_2;-z) + \sum_{\substack{m,n,q \\ (q,mn)=1}} \frac{d_k(m) d_k(n)}{m^{s_1} n^{s_2}} \frac{q^z}{\phi(q)} \sum_{a|(g,q)} \big((g/a)^z-1\big) \mu(a).
\end{align}
For fixed $s_1,s_2$, the sum on the last line is bounded as $z \rightarrow 0^-$, a claim that follows upon applying the crude bound
$$
\sum_{a|(g,q)} \big((g/a)^z-1\big) \mu(a) \ll_\epsilon g^\epsilon |z|.
$$
Multiplying both sides of \eqref{T_to_G} by $z$ and letting $z \rightarrow 0^-$ establishes \eqref{singularity_comparison} for $\Re s_1, \Re s_2 > 1$. (Really this establishes the claim for $z \rightarrow 0^-$, but it is immediate from the factorizations \eqref{T_factor} and \eqref{F_factor} that both limits do exist.) By analytic continuation we obtain \eqref{singularity_comparison} for all $s_1,s_2 \in \mathcal{O}$.

Hence, substituting the representation \eqref{B_Mellin2} for $\mathcal{B}_k$ into the representation \eqref{MG_TwoMellins} with which we concluded the last step, we obtain
\begin{multline}
\label{MGB_Mellin}
\mathcal{MG}_k(Q;X) - \mathcal{B}_k(Q;X) = \frac{1}{(2\pi i)^3} \int_{(\epsilon)} \int_\mathcal{O} \int_\mathcal{O} X^{s_1+s_2-z} Q^z \widetilde{\mathcal{W}}_3(s_1,s_2;z) T_k(s_1,s_2;z)\, ds_1 ds_2 dz \\ + O(XQ).
\end{multline}
We have shown this for $X \leq Q^{(k+2)/k-\delta}$ unconditionally, and for $X \leq Q^{2-\delta}$ on RH.

\vspace{3mm}
\textbf{Step 3: Evaluating the contour integral \eqref{MGB_Mellin} approximating $\mathcal{MG}_k - \mathcal{B}_k$}

In this last step of the proof we turn to evaluating the integral in \eqref{MGB_Mellin}. The basic idea is to push the contour in $z$ from the line $(\epsilon)$ to a curve extending just beyond the point $z=1$, thereby picking up a residue at $z=1$ and $s_1,s_2=1$, which when evaluated gives the formula in the lemma. In the proof that follows we will seek only a first-order asymptotic formula and ignore lower order terms.

Let $\mathcal{C}_k(Q;X)$ be the integral on the right hand side of \eqref{MGB_Mellin}. We may contract each of the contours $\mathcal{O}$ over which $s_1$ and $s_2$ are integrated to the contour $\mathcal{O}' := \{s: |s-1| = 1/\log X\}$ without encountering any singularities. Hence
$$
\mathcal{C}_k(Q;X) =  \int_{(\epsilon)} \underbrace{\int_{\mathcal{O}\,'}\int_{\mathcal{O}\,'} X^{s_1+s_2-z} Q^z \widetilde{\mathcal{W}}_3(s_1,s_2;z) T_k(s_1,s_2;z) ds_1 ds_2}_{:= c(z)}\,dz.
$$
We recall that $\zeta(s)$ has no zeros in the region $\{\sigma+it: \sigma \geq 1 - B/\log(|t|+2)\}$ for some absolute constant $B$. With this as a motivation we define the contour
$$
\gamma:= \{z = \sigma+it: \sigma = 1 - \tfrac{1}{\log Q} + \tfrac{B}{\log(|t|+2)}\},
$$
traversed as usual in the counter-clockwise direction. Our fuller reasons for choosing such a contour will become clear shortly. Also define $\mathcal{O}\,'':= \{z: |z-1| = 3/\log Q\}$. Let $E$ be the region bounded by the line $(\epsilon)$ and the contour $\gamma$. For sufficiently large $Q$, the curve $\mathcal{O}\,''$ is contained in the interior $E$ and one sees that the only possible singularities of the function $c(z)$ in this region $E$ are contained inside $\mathcal{O}\,''$, by using the analytic continuation of $\widetilde{\mathcal{W}}_3$ in Proposition \ref{Mellin_AC} and that of $T_k$ given by the identity \eqref{T_factor}. Hence, by moving the contour $(\epsilon)$ to the \emph{right} to $\gamma$,
\begin{multline}
\label{C_exp}
\mathcal{C}_k(Q;X) = \int_\gamma \int_{\mathcal{O}'} \int_{\mathcal{O}'} X^{s_1+s_2-z} Q^z \widetilde{\mathcal{W}}_3(s_1,s_2;z) T_k(s_1,s_2;z)\, ds_1 ds_2 dz \\ -\int_{\mathcal{O}\,''} \int_{\mathcal{O}\,'} \int_{\mathcal{O}\,'} X^{s_1+s_2-z} Q^z \widetilde{\mathcal{W}}_3(s_1,s_2;z) T_k(s_1,s_2;z)\, ds_1 ds_2\, dz,
\end{multline}
where the `horizontal parts at infinity' of the contour bounding the region $E$ may be shown to be negligible using the bound for $\widetilde{\mathcal{W}}_3$ in Proposition \ref{W3_decay}. 

The first integral in \eqref{C_exp} can be bounded by noting that for all $z\in \gamma$, we have $1+A \geq \Re z \geq 1- 1/\log Q$, so for $s_1 \in \mathcal{O}\,', s_2 \in \mathcal{O}\,', z\in \gamma$,
$$
X^{s_1+s_2-z} Q^z \ll_\epsilon XQ (\log Q)^\epsilon,
$$
for $Q/\log \log Q \leq X \leq Q^2$. Moreover, for such $s_1, s_2, z$,
$$
\widetilde{\mathcal{W}}_3(s_1,s_2;z) \ll_\ell \frac{1}{1+|\Im z|^\ell}
$$
and using \eqref{T_factor} and the fact that $|1-z| \gg 1/\log Q$ for $z \in \gamma$ we have also for such $s_1, s_2,$ and $z$,
$$
T_k(s_1,s_2;z) \ll |\zeta(1-z)|(\log Q)^{2k} \Big(1 + \frac{1}{|z-1|}\Big)^{k^2-2k+1}.
$$
Hence, integrating on the contours $\mathcal{O}', \mathcal{O}''$ and $\gamma$, the first integral in \eqref{C_exp} is no more in order than
$$
XQ (\log Q)^{k^2-2+\epsilon}.
$$

Turning to the second integral of \eqref{C_exp}, we have for 
$$
s_1 = 1 + \frac{s_1'}{\log Q},\quad s_2 = 1+ \frac{s_2'}{\log Q},\quad z = 1 + \frac{z'}{\log Q},
$$
with $|s_1'|=|s_2'| = 1$ and $|z'|=3$, by an expansion of each term into Laurent series,
\begin{multline*}
X^{s_1+s_2-z} Q^z \widetilde{\mathcal{W}}_3(s_1,s_2;z) T_k(s_1,s_2;z) \\= - XQ e^{(s_1'+s_2' - z')\tfrac{\log X}{\log Q}} e^{z'} \frac{(s_1'-z')^k (s_2'-z')^k}{(z')^2 (s_1')^k (s_2')^k (s_1'+s_2'-z')^{k^2}}\, V_k(1,1;1) (\log Q)^{k^2+2} \\+ O(XQ (\log Q)^{k^2+1}).
\end{multline*}
(Used in the above evaluation is  the fact that $\zeta(0) = -1/2$, together with Proposition \ref{Mellin_AC}.) Substituting this into \eqref{C_exp} and making the change of variables $s_1 \mapsto s_1', s_2 \mapsto s_2', z \mapsto z'$ in the integral that remains, we obtain
$$
\mathcal{C}_k(Q;X) = V_k(1,1;1) P_k\Big(\frac{\log X}{\log Q}\Big) XQ (\log Q)^{k^2-1} + O_\epsilon(XQ (\log Q)^{k^2-2+\epsilon}).
$$
Putting this back in \eqref{MGB_Mellin}, we see that Lemma \ref{off_diagonal_unconditional} will be proved if only we can show that $V_k(1,1;1) = \widetilde{a}_k$. But this is a straightforward (if slightly tedious) matter to check.
\end{proof}

\begin{remark} 
We have made use of the zero free region for $\zeta(s)$ out of convenience rather than necessity. Indeed, by a somewhat more computationally intensive proof, explicitly expanding the residues we have isolated above, a more complicated asymptotic formula for $\mathcal{MG}_k(Q;X) - \mathcal{B}_k(Q;X)$, with a power-saving error term, can be obtained.
\end{remark}

\section{Bounding the error term $\mathcal{EG}_k$}
\label{sec:error}

\subsection{A statement of bounds}
In this section we obtain the following bounds for the error term $\mathcal{EG}_k$.  
\begin{lemma}
\label{Ebound_unconditional}  For any $\epsilon >0$, 
$$
\mathcal{EG}_k(Q;X) \ll_\epsilon X^{1+\epsilon} (X/Q)^{k/2}.
$$
In particular, if $X \le Q^{(k+2)/k -\epsilon}$ then 
$$
\mathcal{EG}_k(Q;X) = o(XQ).
$$
\end{lemma}
\begin{lemma}
\label{Ebound_conditional}
Assume GRH. For any $\epsilon > 0$.
$$
\mathcal{EG}_k(Q;X) \ll_\epsilon X^{2+\epsilon}/Q,
$$
In particular, if $X \le Q^{2-\epsilon}$ then 
$$
\mathcal{EG}_k(Q;X) = o(XQ).
$$ 
\end{lemma}

\begin{remark}
\label{obvious_remark}
Owing to the support of $\Phi$ and $\Psi$, there plainly exists a constant $\eta$ depending on these functions so that if $X \leq \eta Q$, we have $\mathcal{EG}_k(Q;X) = 0.$ So in our proofs below we will only be concerned with the case $X \gg Q$.
\end{remark}

\subsection{Moments of $L$-functions}
In order to prove these results we require first some estimates for moments of $L$-functions.
\begin{proposition}
\label{unconditional_moment}
For $R,S \geq 1$, 
$$
\sum_{q\leq R} \sum_{\substack{\chi \Mod q \\ \chi \neq \chi_0}} \int_0^S |L(1/2+it,\chi)|^{2k}\, dt \ll_{k,\epsilon} (R^2 S + (RS)^{k/2}) (RS)^\epsilon,
$$
for any $\epsilon > 0$.
\end{proposition}

\begin{proof}
{This is an application of the large sieve.  Without the average over $t$, the bound would be a special case of a result of Huxley \cite[Thm. 3]{Hu}, (see also \cite[Thm 7.34]{IwKo}). 
 Since the proof of Proposition \ref{unconditional_moment} is much the same as Huxley's we only give a sketch of the proof. By using the approximate functional equation, for $\chi$ primitive we may approximate $L(1/2+it,\chi)^k$ by Dirichlet polynomials of length roughly $(RS)^{k/2}$. By using the hybrid large sieve \cite{Ga} of Gallagher (or see \cite[Thm. 7.17]{IwKo}), we obtain
$$
\sum_{q\leq R} \sum_{\substack{\chi \Mod q \\ \chi \;\mathrm{prim.}}} \int_0^S |L(1/2+it,\chi)|^{2k}\, dt \ll_{k,\epsilon} (R^2 S + (RS)^{k/2}) (RS)^\epsilon,
$$
It is then a simple matter to extend the sum to all non-zero characters.  See, for example, the proof of Proposition 2.9 of \cite{ChLi}.
}
\end{proof}

\begin{proposition}
\label{conditional_moments}
Assume GRH. For $S \geq 2$ and all positive integers $q$,
$$
\sum_{\substack{\chi \Mod q \\ \chi \neq \chi_0}} \int_0^S |L(1/2+it,\chi)|^{2k} \,dt \ll_{k,\epsilon} \phi(q) S (\log qS)^{k^2+\epsilon}.
$$
\end{proposition}

\begin{proof}
This is \cite[Prop. 2.9]{ChLi}.
\end{proof}

\begin{remark}
The result there is proven making use of the conditional moment bounds from \cite{So}.  Arguing as in the work of Harper \cite{Ha}, the $\epsilon$ in the exponent may be removed.  
In fact, we will only need the crude bound of $(QS)^{1+\epsilon}$, which follows from the generalized Lindel\"of hypothesis.
\end{remark}

\subsection{Proofs of the bounds}

\begin{proof}[Proof of Lemma \ref{Ebound_unconditional}]
By using \eqref{EG_sum} we see that
\begin{multline}
\label{E_exp1}
\mathcal{EG}_k(Q;X) = \sum_{(M,N)=1} \sum_{\substack{a,r,g \\ a|g}} \frac{\mu(a)}{\phi(ar)} \sum_{\substack{\chi \Mod{ar}\\\chi \neq \chi_0}} \chi(M)\overline{\chi(N)} d_k(gM) d_k(gN) \mathcal{W}\Big(\frac{gM}{X}, \frac{gN}{X}; \frac{X}{grQ}\Big),
\end{multline}
where from the support of the function $\mathcal{W}$, we have that $gM$, and $gN$ are $\ll X$, and that $r \ll |M-N| /Q \ll X/g$.   
By writing $g = a\lambda$ and using the fact that for any function $\alpha(M,N),$
$$
\sum_{(M,N)=1} \chi(M)\overline{\chi(N)} \alpha(M,N) = \sum_{(j,ar)=1} \sum_{\mm, \nn \geq 1} \mu(j) \chi(\mm) \overline{\chi(\nn)} \alpha(j\mm, j\nn),
$$
where $\chi$ is a character modulo $ar$, we can further expand \eqref{E_exp1} to
$$
\sum_{\mm,\nn\geq 1} \mathop{\sum_{a,\lambda, j \ll X} \sum_{r \ll X/Q}}_{(j,ar)=1} \frac{\mu(a)}{\phi(ar)} \chi(\mm) \overline{\chi(\nn)} \mu(j) d_k(a\lambda j \mm) d_k(a\lambda j \nn) \mathcal{W}\Big(\frac{a\lambda j \mm}{X}, \frac{a\lambda j \nn}{X}; \frac{X}{a\lambda rQ}\Big).
$$
By taking a Mellin transform in the first two variables of $\mathcal{W}$, it follows that for $\beta > 1$,
\begin{multline}
\label{E_exp2}
\mathcal{EG}_k(Q;X) = \mathop{\sum_{a\lambda j \ll X} \sum_{r \ll X/(Qa\lambda)}}_{(j,ar)=1} \frac{\mu(a)\mu(j)}{\phi(ar)} \sum_{\substack{\chi \Mod{ar} \\ \chi \neq \chi_0}} \frac{1}{(2\pi i)^2} \int_{(\beta)} \int_{(\beta)} \frac{X^{s_1+s_2}}{(a\lambda j)^{s_1+s_2}}\\ \times  F_{a\lambda j} (s_1, \chi) F_{a\lambda j} (s_2, \overline{\chi}) \widetilde{\mathcal{W}_2}\Big(s_1,s_2; \frac{X}{a\lambda r Q}\Big)\, ds_1 ds_2,
\end{multline}
where  
$$
F_A(s,\chi):= \sum_{m \geq 1} d_k(Am) \frac{\chi(m)}{m^s}
 = \frac{\prod_{p^\ell || A} \sum_{i\geq 0} d_k(p^{\ell+i})\chi(p^i)p^{-is}} {\prod_{p | A} \sum_{i\geq 0} d_k(p^{i})\chi(p^i)p^{-is}} L(s,\chi)^k.
$$
Since the characters $\chi$ in \eqref{E_exp2} are not principal,  we may shift the contours from the lines $(\beta)$ to the lines $(1/2)$ without encountering any singularities. (Proposition \ref{W2_bound}, which estimates the function $\widetilde{W}_2$, allows us bound the horizontal part of the contour shift.)  Since $d_k(p^{\ell+i}) \leq d_k(p^\ell) d_k(p^i)$, on the line $\Re s = 1/2$, we have 
$$
\Big| \sum_{i\geq 0} d_k(p^{\ell+i})\chi(p^i)p^{-is} \Big| \leq d_k(p^\ell) \sum_{i\geq 0} \frac{d_k(p^i)}{ p^{i/2}} = d_k(p^\ell) \frac{1}{(1-1/\sqrt{p})^k}.
$$
Further 
$$
\Big| \sum_{i \geq 0} d_k(p^{i})\chi(p^i)p^{-is} \Big|  = \Big| 1-\frac{\chi(p)}{p^{s}} \Big|^{-k} \geq \frac{1}{(1+1/\sqrt{p})^k}, 
$$
and so for $\Re s = 1/2$,
$$
|F_A(s,\chi)| \leq d_k(A) \prod_{p|A} \Big(\frac{1+1/\sqrt{p}}{1-1/\sqrt{p}}\Big)^k |L(s,\chi)|^k  \ll_\epsilon A^\epsilon |L(s,\chi)|^k,
$$
for any $\epsilon > 0$.  
  Hence, returning to \eqref{E_exp2} and using also $\phi(ar) \gg (ar)^{1-\epsilon}$, 
\begin{multline*}
\mathcal{EG}_k(Q;X) \ll_\epsilon \mathop{\sum_{a\lambda  j \ll X} \sum_{r \ll X/(Qa\lambda)}}_{(j,ar)=1} \frac{X^{1+\epsilon} }{a^2 \lambda j r} \\
\times \sum_{\substack{\chi \Mod{ar} \\ \chi \neq \chi_0}} \int_{-\infty}^\infty \int_{-\infty}^\infty |L(s_1,\chi)|^k |L(s_2,\chi)|^k \Big|\widetilde{\mathcal{W}}_2\Big(s_1,s_2; \frac{X}{a\lambda r Q}\Big)\Big|\,dt_1 dt_2,
\end{multline*}
where $s_1 = 1/2+it_1$, $s_2 = 1/2+it_2$ above. 
Writing $ar=h$, and estimating the sum over the free variable $j$, we may simplify our bound to 
\begin{equation} 
\label{E_exp3} 
{\mathcal EG}_k(Q;X) \ll  \sum_{h \lambda \ll X/Q} \frac{X^{1+\epsilon}}{h\lambda} \sum_{\substack{\chi \Mod h \\ \chi\neq{\chi_0}}}  \int_{-\infty}^\infty \int_{-\infty}^\infty |L(s_1,\chi)|^k |L(s_2,\chi)|^k \Big|\widetilde{\mathcal{W}}_2\Big(s_1,s_2; \frac{X}{h \lambda Q}\Big)\Big|\,dt_1 dt_2 . 
\end{equation} 
%
We bound this quantity by breaking $h,\lambda$ and $t_1,t_2$ into dyadic intervals, making use of Proposition \ref{W2_bound} to bound $\widetilde{\mathcal{W}}_2$ and Proposition \ref{unconditional_moment} to bound moments of $L$-functions.

{Define $E(S):=\{t: S \leq t \leq 2S\}$ for $S > 1$ and $E(S) = [-1,1]$ for $S=1$.  We have for $S_1, S_2 \geq 1$ with $ S_2 \geq S_1$, and for $H$ and $L$ with $HL \ll X/Q$, 
\begin{multline}
\label{dyadic_bound}
\sum_{\substack{H \leq h \leq 2H \\ L \leq \lambda \leq 2L}} \frac{1}{\lambda h}
 \sum_{\substack{\chi \Mod h \\ \chi \neq \chi_0}} \int_{E(S_2)} \int_{E(S_1)} \big( |L(s_1,\chi)|^{2k} + |L(s_2,\chi)|^{2k}\big) \Big|\widetilde{\mathcal{W}}_2\Big(s_1,s_2; \frac{X}{h \lambda Q}\Big)\Big|\,dt_1 dt_2 \\
\ll_{p,\epsilon} \frac{1}{LH} (H^2 S_2 + (HS_2)^{k/2}) (HS_2)^\epsilon S_1 \frac{(X/HLQ)^{p-1}}{S_2^p} .
\end{multline}
We make use of this bound with $p=1$ for $S_2 \leq X/HLQ$ and $p = k/2+2$ otherwise. We return to \eqref{E_exp3}. From the symmetry between $s_1$ and $s_2$, and since $|L(s_1,\chi) L(s_2,\chi)|^k \ll |L(s_1,\chi)|^{2k} + |L(s_2,\chi)|^{2k}$, 
\begin{multline*}
\mathcal{EG}_k(Q;X) \ll_\epsilon X^{1+\epsilon} \sumd_{\substack{H,L \\HL \ll X/Q}} \frac{1}{L H}\bigg(\sumd_{\substack{S_1 \leq S_2\\ S_2 \leq X/HLQ}} (H^2 S_2 + (HS_2)^{k/2})(HS_2)^\epsilon S_1 \frac{1}{S_2} \\
+ \sumd_{\substack{S_1 \leq S_2 \\ S_2 > X/HLQ}}(H^2 S_2 + (HS_2)^{k/2})(HS_2)^\epsilon S_1 \frac{1}{S_2^{k/2+2}}\Big(\frac{X}{HLQ}\Big)^{k/2+1}\bigg),
\end{multline*}
where $\sumd$ indicates a dyadic sum over powers of $2$. By estimating these sums we obtain
$$
\mathcal{EG}_k(Q;X) \ll_\epsilon X^{1+\epsilon}(X/Q) + X^{1+\epsilon} (X/Q)^{k/2+\epsilon}.
$$
This is $O_\epsilon(X^{1+\epsilon} (X/Q)^{k/2})$, establishing the claim.}
\end{proof}

\begin{proof}[Proof of Lemma \ref{Ebound_conditional}]
We proceed almost as above, but instead of making use of the large sieve bound in Proposition \ref{unconditional_moment} we use the GRH bound of Proposition \ref{conditional_moments}. Applying this to the dyadic intervals in \eqref{dyadic_bound}, we obtain for the upper bound in the second line of that inequality,
$$
\ll_{p,\epsilon} \frac{1}{LH} H^2 S_2 (HS_2)^\epsilon S_1 \frac{(X/HLQ)^{p-1}}{S_2^p} .
$$
Using this in \eqref{E_exp3}, and evaluating the dyadic sum as before, we see
$$
\mathcal{EG}_k(Q;X) \ll_\epsilon X^{1+\epsilon}(X/Q),
$$
as claimed.
\end{proof}

\section{An asymptotic formula for $\Delta_k$}
\label{sec:final}

We summarize what we have shown. If $X^\delta \leq X \leq Q/\log \log Q)$, we have found an asymptotic formula for $\Delta_k(Q;X)$ in Lemma \ref{small_c}, namely
$$
\Delta_k(Q;X) \sim \widetilde{a}_k \frac{c^{k^2-1}}{(k^2-1)!} QX (\log Q)^{k^2-1},
$$
where $c:= \log X/\log Q$. 

On the other hand, if $Q/\log \log Q \leq X \leq Q^{\tfrac{k+2}{k}-\delta}$, from Lemmas \ref{diagonal}, \ref{off_diagonal_unconditional}, and \ref{Ebound_unconditional},
\begin{align*}
\Delta_k(Q;X) =& \mathcal{D}_k(Q;X) - \mathcal{B}_k(Q;X) + \mathcal{MG}_k + \mathcal{EG}_k\\
=& \widetilde{a}_k \Big(\frac{c^{k^2-1}}{(k^2-1)!} + P_k(c)\Big) QX (\log Q)^{k^2-1} + O_\epsilon(XQ (\log Q)^{k^2-2+\epsilon}),
\end{align*}
where recall $P_k$ is defined by \eqref{P_k}. Likewise from Lemmas \ref{diagonal}, \ref{off_diagonal_unconditional}, and \ref{Ebound_conditional}, on GRH this is true for $Q/\log \log Q \leq X \leq Q^{2-\delta}$.

Thus Theorems \ref{main_thm} and \ref{main_thm_GRH} will be established if only we demonstrate that
\begin{equation}
\label{gamma_k_01_prelim}
\gamma_k(c)  = c^{k^2-1}/(k^2-1)!,\quad \mathrm{for} \quad c \in (0,1],
\end{equation}
and
\begin{equation}
\label{gamma_k_12_prelim}
\gamma_k(c) = c^{k^2-1}/(k^2-1)! + P_k(c),\quad \mathrm{for} \quad c \in [1,2).
\end{equation}
The relation \eqref{gamma_k_01_prelim} is established in section 4.4.3 of \cite{KeRoRoRu}, and in the next section we prove 
\eqref{gamma_k_12_prelim}.  


\section{A comparison with $\gamma_k$ and remarks on $P_k$}
\label{sec:comparison}

\subsection{A RMT computation}
Our purpose is to demonstrate the identity 
\eqref{gamma_k_12_prelim}.
	


\begin{lemma}
	\label{gamma_k_poly}
	For $c\in [1,2)$,
	$$
	\gamma_k(c) = \frac{c^{k^2-1}}{(k^2-1)!} + P_k(c),
	$$
	where $P_k(c)$ is the polynomial given by \eqref{P_k}.
\end{lemma}

It will be convenient for us to deal with $P_k$ as a multiple residue, rather than a more traditionally written polynomial (though we do record a more traditional expression for it in the next subsection). 

We note that \eqref{P_k} may be rewritten as
\begin{equation}
\label{P_k2}
P_k(c) = - \Res_{z'=0} \Res_{\substack{s_1=0 \\ s_2=0}} e^{cz'} e^{s_1+s_2-z'}\frac{(s_1-z')^k (s_2-z')^k}{(z')^{k^2}s_1^k s_2^k (s_1+s_2-z)^2}.
\end{equation}
We outline a proof of this identity, leaving details to the reader. Begin by writing \eqref{P_k} as a multiple contour integral, with the contour in the $z$ variable containing in its interior the sumset of the contours over $s_1$ and $s_2$. 
It is now possible to swap the order of integration, and make a change of variable $z' = s_1 + s_2 -z$, and then swap the order of integration again.  The right side of  \eqref{P_k2} is the resulting residue. This change of variables will prove useful for us later.

Though the piecewise polynomial $\gamma_k(c)$ is succinctly expressed by the integral \eqref{piecewise_poly}, in proving Lemma \ref{gamma_k_poly} it will be useful to return to $\gamma_k$'s origins in random matrix theory over the unitary group. As usual, we let $U(N)$ be the $N\times N$ group of unitary matrices, endowed with Haar probability measure $dg$.

By expanding characteristic polynomials, it is apparent that the integral
$$
\mathcal{I}:= \int_{U(N)} \det(1-x g)^k \det(1-g^{-1})\, dg
$$
is a polynomial of degree $kN$; that is
$$
\mathcal{I} = \sum_{m=0}^{kN} I_k(m;N) x^m,
$$
with coefficients $I_k(m;N)$ defined by this relation. In \cite{KeRoRoRu}, it is shown (Theorem 1.5) that as $N\rightarrow \infty$,
\begin{equation}
\label{I_K_asymptotic}
I_k(m;N) = \gamma_k(m/N) N^{k^2-1} + O(N^{k^2-2}),
\end{equation}
the estimate uniform in $m$.

We will prove Lemma \ref{gamma_k_poly} by finding a formula for $I_k(m;N)$ and taking the limit $N\rightarrow \infty$. A formula of \cite{CoFaKeRuSn} is the starting point. 

\begin{theorem}[Conrey-Farmer-Keating-Rubinstein-Snaith]
\label{CFKRS}
For $A$ and $B$ finite collections of complex numbers,
\begin{equation}
\label{autocorrelation}
\int_{U(N)} \prod_{\alpha\in A} \det(1-\alpha g) \prod_{\beta \in B} \det(1-\beta g^{-1})\,dg = \sum_{\substack{S \subseteq A \\ T \subseteq B  \\ |S| = |T|}} \prod_{\alpha\in S} \alpha^N \prod_{\beta\in T} \beta^N Z(\overline{S}+ T^-\,,\, \overline{T} + S^-),
\end{equation}
where $\overline{S} := A\setminus S$, $\overline{T} := B\setminus T$, $S^- := \{-\alpha:\, \alpha \in S\}$, $T^- := \{-\beta:\, \beta \in T\}$, and for any finite collections of complex numbers $A$ and $B$,
$$
Z(A,B):= \prod_{\substack{\alpha\in A\\ \beta \in B}} \frac{1}{1-\alpha\beta}.
$$
\end{theorem}

\begin{remark}
Following the convention of this area, we use the notation $\overline{S}+T^{-}$ to mean appending the list $\overline{S}$ to the list $T^{-}$. So $\{1,4,4\}+\{2,4,5\} = \{1,2,4,4,4,5\}$.
\end{remark}

\begin{remark}
The singularities that might at first appear to occur on the right hand side of \eqref{autocorrelation} are known to be removable. Terms in this sum that grow arbitrarily large for certain choices of $A$ and $B$ cancel out with other terms for such $A,B$. (Indeed, this is evident from the left hand side.)
\end{remark}

\begin{remark}
The method of proof we take closely resembles that of section 4.2 of \cite{KeRoRoRu}, though our approach differs somewhat in its specifics, since our goal is to match up with the residue defining $P_k$.
\end{remark}

\begin{proof}[Proof of Lemma \ref{gamma_k_poly}]
Our proof proceeds in two steps.

\emph{Step 1:} By making use of Theorem \ref{CFKRS}, we will represent $I_k(m;N)$ as a residue: for $N < m < 2N$,
\begin{equation}
\label{step1.1}
I_k(m;N) = -\Res_{\zeta=1} \zeta^{-(m+1)} f_N(\zeta),
\end{equation}
with
\begin{multline}
\label{step1.2}
f_N(\zeta):= \frac{1}{(1-\zeta)^{k^2}} - \frac{1}{(2\pi i)^2} \int_{|s_1|=1/10} \int_{|s_2|=1/10} \frac{\zeta^{N+1} e^{(N+1)(s_1+s_2)}}{(1-\zeta)^{k^2}} \Big(\frac{1-\zeta e^{s_2}}{1-e^{-s_1}}\Big)^k \\
\times \Big(\frac{1-\zeta e^{s_1}}{1-e^{-s_2}}\Big)^k \frac{ds_1 ds_2}{(1-\zeta e^{s_1+s_2})^2}.
\end{multline}

\emph{Step 2:} By analyzing this residue in the limit $m,N \rightarrow\infty$ with $m/N \rightarrow c \in [1,2)$, and using
\begin{equation}
\label{step2.0}
\gamma_k(c) = \lim \frac{1}{N^{k^2-1}} I_k(m;N),
\end{equation}
which clearly follows from \eqref{I_K_asymptotic}, we are able to verify Lemma \ref{gamma_k_poly}.

We turn to \textbf{Step 1}. We introduce the polynomial
\begin{align*}
\mathcal{I}^{(\epsilon,\delta)}&:= \int_{U(N)} \prod_{i=1}^k \det(1-x e^{\epsilon_i} g) \prod_{j=1}^k \det(1-e^{\delta_i} g^{-1})\, dg \\
&= \sum_{m=0}^{kN} I_k^{(\epsilon,\delta)}(m;N) x^m,
\end{align*}
with the coefficients $I_k^{(\epsilon,\delta)}(m;N)$ defined by this relation. Note that for fixed $m$ and $N$,
\begin{equation}
\label{I_limit}
\lim_{\epsilon, \delta \rightarrow 0} I_k^{(\epsilon,\delta)}(m;N) = I_k(m;N).
\end{equation}
By applying Theorem \ref{CFKRS}, for $\epsilon_i$ and $\delta_j$ all distinct (distinct so that we need not worry about the analysis of any removable singularities in the right-hand side of \eqref{autocorrelation}) we have
\begin{equation}
\label{diag_and_offdiag}
\mathcal{I}^{(\epsilon,\delta)} = A_1^{(\epsilon,\delta)}(x) + A_2^{(\epsilon,\delta)}(x) + O_{\epsilon,\delta}(x^{2N}),
\end{equation}
for sufficiently small $x$, where
$$
A_1^{(\epsilon,\delta)}(x):= \prod_{i,j=1}^k \frac{1}{1-x e^{\epsilon_i+\delta_j}},
$$
which corresponds to the term $S=T=\emptyset$ in the summation on the right hand side of \eqref{autocorrelation}, and
\begin{align*}
A_2^{(\epsilon,\delta)}(x) :=& \sum_{\mu,\nu=1}^k x^N e^{N(\epsilon_\mu+ \delta_\nu)} \Big(\prod_{\substack{i\neq \mu\\ j \neq \nu}} \frac{1}{1-x e^{\epsilon_i + \delta_j}}\Big)  \Big(\prod_{i\neq \mu} \frac{1}{1-e^{-\epsilon_\mu+\epsilon_i}}\Big) \\ &\hspace{60mm}\times\Big(\prod_{j\neq \nu} \frac{1}{1-e^{-\delta_\nu + \delta_j}}\Big) \Big(\frac{1}{1-x^{-1}e^{-\epsilon_\mu-\delta_\nu}}\Big) \\
=& - \sum_{\mu,\nu=1}^k x^{N+1} e^{(N+1)(\epsilon_\mu+\delta_\nu)} \Big(\prod_{i,j} \frac{1}{1-x e^{\epsilon_i + \delta_j}}\Big)\Big(\prod_{i\neq \mu} \frac{1-x e^{\epsilon_i +\delta_\nu}}{1-e^{-\epsilon_\mu+\epsilon_i}}\Big) \\ &\hspace{60mm}\times \Big(\prod_{j\neq \nu} \frac{1-x x^{\delta_j +\epsilon_\mu}}{1-e^{-\delta_\nu + \delta_j}}\Big)\Big(\frac{1}{1-x e^{\epsilon_\mu+\delta_\nu}}\Big)^2
\end{align*}
which corresponds to the terms for which $|S|=|T|=1$ in the summation of \eqref{autocorrelation}; these terms arise from $S = \{\epsilon_\mu\}$ and $T=\{\delta_\nu\}$ for $1\leq \mu, \nu \leq k$. The error term $O_{\epsilon,\delta}(x^{2N})$ in \eqref{diag_and_offdiag} is $x^{2N} g_{\epsilon,\delta}(x)$, where $g_{\epsilon,\delta}(x)$ is an analytic function in $x$ around the origin, arising because all remaining terms in the summation in \eqref{autocorrelation} have $|S| \geq 2$ and this produces for $j \geq 2N$ a summand with $x^j$ multiplied by a rational function analytic around $x=0$.

We note that $A_2$ has a further simplification. For $\epsilon_i$ and $\delta_j$ always distinct and $|\epsilon_i|, |\delta_j| \leq 1/10$ for all $i,j$, the reader should check that we have by a computation of residues
\begin{multline}
\label{A_2}
A_2^{(\epsilon,\delta)}(x) = - \frac{1}{(2\pi i)^2} \int_{|s_1|=1/10}\int_{|s_2|=1/10} \frac{x^{N+1} e^{(N+1)(s_1+s_2)}}{\prod_{i,j} (1-x e^{\epsilon_i+\delta_j})} \Big(\prod_i \frac{1- x e^{\epsilon_i + s_2}}{1- e^{\epsilon_i - s_1}} \Big)\\
\times \Big(\prod_j \frac{1 - x e^{\delta_j +s_1}}{1- e^{\delta_j -s_2}} \Big)\frac{1}{(1- x e^{s_1+s_2})^2}\, ds_1 ds_2,
\end{multline}
so long as $|x|$ is sufficiently small that the final term $(1- x e^{s_1+s_2})^{-2}$ contributes no residue ($|x| < e^{-2/10}$ suffices). (Thus the poles inside this contour are exactly those at $s_1 = \epsilon_1,...,\epsilon_k$ and $s_2 = \delta_1,...,\delta_k$. Taking contours of radius, say, $1/10$ ensures that e.g. $\epsilon_1 + 2\pi i$ is not also a pole.)
	
From \eqref{diag_and_offdiag} then for $m\leq 2N$,
$$
I_k^{(\epsilon,\delta)}(m;N) = \frac{1}{2\pi i} \int_{|\zeta|=R} \zeta^{-(m+1)} \big(A_1^{(\epsilon,\delta)}(\zeta) + A_2^{(\epsilon,\delta)}(\zeta)\big)\,d\zeta,
$$
with $R$ a parameter chosen sufficiently small (i.e. $R < e^{-2/10}$). Taking a limit as $\epsilon,\delta\rightarrow 0$ along paths such that all $\epsilon_i, \delta_j$ remain distinct, we see from \eqref{I_limit} and uniform convergence (both on the contour $|\zeta|=R$ and in the integral representation \eqref{A_2} for $A_2^{(\epsilon,\delta)}$) that
\begin{equation}
\label{contour_I}
I_k(m;N) = \frac{1}{2\pi i} \int_{|\zeta|=R} \zeta^{-(m+1)} f_N(\zeta)\,d\zeta,
\end{equation}
for $m < 2N$, where $f_N(\zeta)$ is defined by \eqref{step1.2}.
	
By inspection, one sees that $f_N(\zeta)$ is meromorphic with a singularity only at $\zeta=1$, and for $\zeta$ of large moduli $|f_N(\zeta)| = O(|\zeta|^N).$ Hence for $N < m < 2N$ and $R > 1$,
\begin{align*}
I_k(m;N) &= \frac{1}{2\pi i} \int_{|\zeta|=R} \zeta^{-(m+1)} f_N(\zeta)\,d\zeta - \Res_{\zeta = 1} \zeta^{-(m+1)} f_N(\zeta) \\
& = - \Res_{\zeta=1} \zeta^{-(m+1)}f_N(\zeta),
\end{align*}
with the second line following by letting $R\rightarrow\infty$. This proves the claim of Step 1.

We turn to \textbf{Step 2}. We let the residue just obtained in Step 1 be the result of a contour integral; making the change of variable $z = 1-\zeta$, we have for $N < m < 2M$,
\begin{equation}
\label{step2.1}
I_k(m;N) = \frac{1}{2\pi i} \int_{|z| = r} (1-z)^{-(m+1)} f_N(1-z)\,dz,
\end{equation}
for any $r < 1$. 

Note that the integrand in the definition \eqref{step1.2} of $f_N(\zeta)$ is analytic in $s_1$ for $s_1\neq 0$ in the disc $|s_1|<1/10$, and likewise for $s_2$. We can therefore replace the contours $|s_1|=1/10$, $s_2| = 1/10$ by $|s_1| = \rho$, $|s_2| = \rho$ for any $\rho < 1/10$. We set $r = 1/N$ in \eqref{step2.1} and take the contours $|s_1|=1/N$, $|s_2|=1/N$ in the representation \eqref{step1.2} of $f_N$. By making the re-scaling change of variables $z' = Nz$, $s_1' = N s_1$, $s_2' = Ns_2$, we see that
\begin{align*}
I_k(m;N) =& \frac{N^{k^2-1}}{2\pi i} \int_{|z'|=1} \frac{1}{\Big(1- \frac{z'}{N}\Big)^{(m+1)}} \Bigg[ \frac{1}{(z')^{k^2}} - \frac{1}{2\pi i} \int_{|s_1'| = 1} \int _{|s_2'|=1} \frac{\Big(1-\frac{z'}{N}\Big)^{(N+1)} e^{(1+1/N)(s_1+s_2)}}{(z')^{k^2}}\\
&\quad \times \Big(\frac{s_1'-z'}{s_2'} + O(1/N)\Big)^k \Big(\frac{s_2'-z'}{s_1'} + O(1/N)\Big)^k \frac{ds_1' ds_2'}{(s_1' + s_2' - z' + O(1/N))^2}\Bigg]\,dz'.
\end{align*}
If $m,N\rightarrow\infty$ in such a way that $m/N \rightarrow c$, then by uniform convergence,
\begin{multline}
\label{nearly_done}
\lim \frac{I_k(m;N)}{N^{k^2-1}} = \frac{1}{2\pi i} \int_{|z'|=1} e^{cz'} \Bigg[ \frac{1}{(z')^{k^2}}\, -\, \frac{1}{(2\pi i)^2}\int_{|s_1'|=1}\int_{|s_2'|=1} \frac{e^{s_1'+s_2'-z'}}{(z')^{k^2}}\\
\times \frac{(s_1'-z')^k (s_2'-z')^k}{(s_1')^k (s_2')^k (s_1'+s_2'-z)^2}\, ds_1'ds_2'\Bigg]\,dz'. 
\end{multline}
Yet from \eqref{step2.0} the left hand side of \eqref{nearly_done} is $\gamma_k(c)$, while on the right hand side we note
$$
\frac{1}{2\pi i} \int_{|z'|=1} \frac{e^{cz'}}{(z')^{k^2}}\,dz' = \frac{c^{k^2-1}}{(k^2-1)!}.
$$
Using \eqref{P_k2} to extract $P_k(c)$ from what remains, we  have verified the lemma.
\end{proof}

\subsection{An expansions of $P_k$}

Though it is was not necessary for our purposes, we record an expansion of the polynomial $P_k(c)$ in a more traditional form that may be of interest. 

\begin{proposition}
\label{P_k_expand}  With $\binom{m}{i_1,i_2,i_3}$ denoting the multinomial coefficient, we have 
\begin{multline*}
P_k(c) = \frac{(-1)^k}{(k^2-1)!}\sum_{\substack{a,b \geq 0 \\ \alpha,\beta \geq 0}} (-1)^{a+b+\alpha+\beta} c^{a+b}(1-c)^{k^2-1-a-b} \binom{k^2-1}{a,b,k^2-1-a-b} \\ \times \binom{k^2-1+\alpha+\beta}  {\alpha,\beta,k^2-1}\binom{k}{a+\alpha+1}\binom{k}{b+\beta+1}.
\end{multline*}
\end{proposition}

The computation follows (with some book-keeping) by writing
$$
P_k(c) = \Res_{z=0} \frac{(-1)^{k^2}}{z^{k^2}} e^{(1-c)z} \Res_{\substack{s_1 = 0 \\ s_2 = 0}} e^{c s_1} e^{c s_2} \Big(1-\frac{z}{s_1}\Big)^k \Big(1-\frac{z}{s_2}\Big)^k\Big(1-\frac{s_1+s_2}{z}\Big)^{-k^2},
$$
and, in order to compute residues, expanding each of $(1-z/s_1)^k$, $(1-z/s_2)^k$, and $(1-(s_1+s_2)/z)^{-k^2}$ into powers of $z/s_1$, $z/s_2$ and $(s_1+s_2)/z$ respectively.

\end{document}